\documentclass[11pt,reqno]{amsart}
\usepackage{fullpage}

\usepackage{graphicx}
\usepackage[draft]{hyperref}
\usepackage{amsmath,amsopn,amssymb,amsfonts,stmaryrd}
\usepackage{verbatim}
\usepackage{amsthm}
\usepackage{mathtools}
\usepackage{color}
\usepackage[x11names]{xcolor}
\usepackage{enumitem}
\usepackage[framemethod=TikZ]{mdframed}
\usepackage{bbm}
\usepackage{mathrsfs}
\usepackage{booktabs}
\usepackage{caption}
\usepackage{bm}
\usepackage{tensor}
\usepackage{cancel}
\usepackage{framed}
\colorlet{shadecolor}{LavenderBlush3}

\usepackage{mathrsfs}
\usepackage[scr=boondox]{mathalfa}

\newcommand{\R}{\mathbb R}
\newcommand{\IR}{{\mathbb R}}
\newcommand{\IC}{{\mathbb C}}
\newcommand{\IZ}{{\mathbb Z}}
\newcommand{\Z}{\mathbb{Z}}
\newcommand{\IN}{{\mathbb N}}
\newcommand{\IQ}{{\mathbb Q}}
\newcommand{\IH}{{\mathbb H}}
\newcommand{\C}{{\mathbb C}}
\newcommand{\Q}{{\mathbb Q}}

\newcommand{\CE}{\mathcal{E}}%
\newcommand{\CF}{\mathcal{F}}

\newcommand{\calT}{\mathscr{T}}

\newcommand{\sgn}{\mbox{sgn}}

\theoremstyle{plain}
\newtheorem{thm}{Theorem}[section]
\newtheorem{cor}[thm]{Corollary}
\newtheorem{lem}[thm]{Lemma}

\newtheorem{prop}[thm]{Proposition}

\newtheorem*{remark}{Remark}

\theoremstyle{definition}

\numberwithin{equation}{section}

\newcommand{\pmat}[1]{\left( \smallmatrix #1 \endsmallmatrix \right)}
\newcommand{\mat}[1]{\left( \begin{matrix} #1 \end{matrix} \right)}

\renewcommand{\sgn}{\textnormal{sgn}}

\def\lp{\left(}
\def\rp{\right)}

\def\a{\alpha}

\def\d{\delta}

\def\k{\kappa}
\def\l{\lambda}

\def\z{\zeta}
\def\th{\theta}

\def\e{\varepsilon}

\def\n{\nu}

\def\t{\tau}

\newcommand{\re}{{\rm Re}}

\renewcommand{\sgn}{{\rm sgn}}

\def\wt{\widetilde}
\def\wh{\widehat}

\newcommand{\andd}{\quad \mbox{ and } \quad}

\setlist[itemize]{noitemsep, topsep=0pt}

\makeatletter
\newcommand{\vast}{\bBigg@{4}}
\newcommand{\Vast}{\bBigg@{5}}
\makeatother

\makeatletter
\newcommand{\nolisttopbreak}{\par\nobreak\@afterheading}
\makeatother

\renewcommand{\pmod}[1]{\ ( \mathrm{mod} \, #1 )}

\newcommand{\psum}{\sideset{}{^*}\sum}

\newcommand{\Log}{\mathrm{Log}}
\renewcommand{\calT}{\mathcal{T}}

\allowdisplaybreaks

\title{Asymptotics of partition parts in arithmetic progressions}

\author{Kathrin Bringmann}
\author{Caner Nazaroglu}
\address{University of Cologne, Department of Mathematics and Computer Science, Weyertal 86-90, 50931 Cologne, Germany}
\email{kbringma@uni-koeln.de}
\email{cnazarog@uni-koeln.de}

\author{Jan-Willem M. van Ittersum}
\address{Department of Mathematics and Computer Science, University of Cologne,
	Weyertal 86-90, 50931 Cologne, Germany}
\curraddr{Korteweg--de Vries Institute for Mathematics, University of Amsterdam, Postbus 94248, 1090 GE  Amsterdam, The Netherlands}
\email{j.w.m.vanittersum@uva.nl}

\makeatletter
\@namedef{subjclassname@2020}{%
\textup{2020} Mathematics Subject Classification}
\makeatother


\begin{document}
\maketitle

\begin{abstract}
We study the distribution of partition parts in arithmetic progressions and find asymptotic results that capture all exponentially growing terms. This is accomplished by studying the behavior of non-modular Eisenstein series that appear in their generating function and have expressions in terms of indefinite and false-indefinite theta functions.
\end{abstract}

\section{Introduction and Statement of results}\label{sec:introduction}
Distribution results in arithmetic progressions concern how arithmetic objects--such as numbers, Fourier coefficients of modular forms, or counts of combinatorial quantities--are distributed among residue classes$\pmod{R}$. A central theme is understanding whether these objects are equidistributed among the coprime classes$\pmod{R}$ and, if not, quantifying the bias. For example, Dirichlet’s Theorem guarantees an infinitude of primes in each residue class, while the Siegel--Walfisz Theorem and the Bombieri--Vinogradov Theorem provide asymptotic formulas and error terms. This paper investigates the asymptotic behavior of the number of partition parts in arithmetic progressions, as well as weighted hook lengths in arithmetic progressions.

Partition parts and multiplicities are related to Eisenstein series in various ways. For example, in \cite{vIsym} it was shown that the generating function associated to a symmetric polynomial in partition parts is a polynomial in Eisenstein series. Another example is the \emph{MacMahon $q$-series}~\cite{Mac1920}
\[
\sum_{n_1>\cdots>n_r>0} \frac{q^{n_1+\ldots+n_r}}{(1-q^{n_1})^2\cdots(1-q^{n_r})^2} \qquad (r\in \mathbb{N}) 
\]
which recently attracted considerable interest, including the prime-detecting properties of its Fourier coefficients \cite{CIO24}. This function, whose $n$-th Fourier coefficient equals the sum of the products of the part multiplicities for partitions of $n$ with $r$ distinct part sizes, is a quasimodular form and admits an explicit expression in terms of Eisenstein series \cite{AR13, Bac24}. Variations, such as
\[
\sum_{\substack{n_1>\cdots>n_r>0\\ n_1,\ldots,n_r \text{ odd}}} \frac{q^{n_1+\ldots+n_r}}{(1-q^{n_1})^2\cdots(1-q^{n_r})^2} \qquad (r\in \mathbb{N})
\]
which also go back to MacMahon, can similarly be expressed in terms of higher level Eisenstein series \cite{NPS25}, such as (for $\tau$ in the upper half plane~$\IH$, $\k, R \in \IN$, and $r \in \{1,2,\ldots,R\}$)
\[
F_{R,r,\k} (\t) := 
\!\!\!\! \sum_{\substack{m \geq 1 \\ m \equiv r \pmod{R}}} \!\!\!\! 
m^{\k-1}\frac{q^m}{1-q^m} = R^{\k-1}\sum_{\substack{n_1\geq1\\n_2\geq0}} \left(n_2+\frac{r}{R}\right)^{\k-1} q^{R n_1 \left(n_2+\frac{r}{R}\right)}  \qquad \left(q:=e^{2\pi i \tau}\right).
\]
These Eisenstein series are typically not quasimodular, but still fit in the framework of indefinite theta functions and their false variants.  They also fit into the theory of higher level multiple Eisenstein series admitting an expression as a lattice sum over positive lattice points \cite{GKZ06, KT13, YZ15}. 

In this paper, we study such non-modular Eisenstein series of weight $\k=1$. We decompose these functions into an indefinite theta function $f_{R,\bm\a}$ and a false-indefinite theta function $g_{R,\bm\a}$ (up to lower rank corrections) as
\begin{equation}\label{eq:FRr_indef_falseindef_decomposition}
F_{R,r}(\t):= F_{R,r,1} (\t) =
f_{R,(0,r)} (\t) + g_{R,(0,r)} (\t) ,
\end{equation}
where, for $\bm\alpha\in \Z^2$,
\begin{align}
\label{eq:f_indef_definition}
f_{R,{\bm \a}} (\t) &:=
\frac{1}{4} \sum_{\substack{{\bm n} \in \IZ^2 + \frac{\bm{\a}}{R}\! \\ n_1,n_2 \neq 0}}
(\sgn (n_1) + \sgn (n_2))  q^{Rn_1n_2}, \\
\notag
g_{R,{\bm \a}} (\t) &:=
\frac{1}{4} \sum_{\substack{{\bm n} \in \IZ^2 + \frac{{\bm \a}}{R}\! \\ n_1,n_2 \neq 0}}
(1+ \sgn (n_1)  \sgn (n_2))  q^{Rn_1n_2} .
\end{align}
Recall that in Zwegers’ work on mock theta functions, Mordell integrals arose as obstructions to modularity \cite[Proposition 1.5]{Zwe02}. Similarly, we define a Mordell-type integral $\mathcal{I}_{R,\bm\a,\varrho}$ which captures both the false-indefinite theta function $g_{R,\bm\a}$ and its obstruction to modularity
\begin{equation}\label{eq:calI_definition}
\mathcal{I}_{R,{\bm \a},\varrho} (\t) :=
\frac{1}{2\pi i} 
\sum_{\substack{{\bm n} \in \IZ^2 + \frac{{\bm \a}}{R}\! \\ n_1,n_2 \neq 0}}
\frac{e^{  {2\pi i R\varrho n_1n_2} }}{Rn_1 n_2} 
\mathrm{PV}\!\int_{0}^\infty\!
\frac{x  e^{2 \pi i \left(\t-\varrho\right) x}}{x\!-\!Rn_1n_2}  dx \qquad \left(\varrho \in \IQ\right),
\end{equation}
where $\mathrm{PV}$ denotes the (Cauchy) principal value of the integral. As our first result, we derive the modular transformation of these functions, that is, an exact expression for $F_{R,r}$ at a cusp $\frac{h}{k}\in \Q$.

\begin{thm}\label{thm:modulartransformation}
Let $R \in \IN$, $r \in \{1,2,\ldots,R\}$, and $h, k \in \mathbb{N}_0$ with $0\leq h<k$ and $\gcd (h,k) =1$. Then, for any $z \in \IC$ with $z_1 > 0$, we have\footnote{
Here and throughout, we write $w\in\C$ as $w=w_1 +iw_2$ with $w_1,w_2 \in \IR$.
}

\begin{multline*}
F_{R,r} \!\lp \frac{h}{k} +  \frac{iz}{k^2} \rp =
- C_{R,r,k} \frac{\Log (z)}{z}
+ \frac{A_{R,r,h,k}}{z} + B_{R,r,h,k}
\\*
+ \frac{ik}{z} \sum_{\a_1,\a_2 \pmod{R}}  \psi_{R,r,h,k} ({\bm \alpha})
\lp f_{R,{\bm \a}} \!\lp \frac{h'}{k} + \frac{i}{z} \rp + \mathcal{I}_{R,{\bm \a},\frac{h'}{k}} \!\lp \frac{h'}{k} + \frac{i}{z} \rp \rp,
\end{multline*}
where $h' \in \IZ$ with $0 \leq h' < k$ and $hh' \equiv -1 \pmod{k}$,  $ \psi_{R,r,h,k} ({\bm \alpha})$ is the Weil multiplier \eqref{eq:psi_definition}, and $A_{R,r,h,k}$, $B_{R,r,h,k}$, and $ C_{R,r,k}$ are constants defined in \eqref{eq:alpha_definition}, \eqref{eq:beta_definition}, and \eqref{eq:gamma_definition}, respectively.  
\end{thm}
\begin{remark}
The antisymmetric difference $F_{R,r}-F_{R,R-r}$ can be expressed purely in terms of~$f_{R,\bm{\a}}$. The indefinite theta function $f_{R,\bm{\a}}$ (up to the addition of a constant term) is indeed a modular form, and the associated modular invariance follows from Theorem~\ref{thm:modulartransformation}. The more intricate form of the theorem reflects the non-modularity of $F_{R,r}$.
\end{remark}

This transformation plays a central role in our study of the following two partition statistics: the number of parts in arithmetic progressions and the reciprocal sum of hook lengths in arithmetic progressions. More precisely, for the first, for $R \in \IN$, $r \in \{1,2,\ldots,R\}$ and a partition $\lambda$, we let
\[
T_{R,r}(\lambda) := |\{\lambda_j: \lambda_j\equiv r \pmod R\}| 
\]
be the number of parts of $\lambda$ in a fixed residue class $r\pmod R$ and we let 
\[
\mathcal{T}_{R,r}(n) := \sum_{\lambda \vdash n} T_{R,r}(\lambda)
\]
be the number of parts $r\pmod R$ in all partitions of $n$ as studied in \cite{BM15,BM17} (this was denoted by $\wh{T}_{r,R} (n)$ there). By \cite[Lemma~2.1]{BM15}, its generating function is given by
\begin{equation}\label{eq:That_generating_function}
\sum_{n\geq 1} \calT_{R,r}(n) q^{n-\frac{1}{24}} = \frac{F_{R,r} (\t)}{\eta (\t)},
\end{equation}
where $\eta (\t) := q^{\frac{1}{24}} \prod_{n \geq 1} (1-q^n)$ is the \textit{Dedekind eta function}.

There is a bias in $\mathcal{T}_{R,r}(n)$ towards smaller values of $r$ (see also \cite{DS05}). In \cite{BM15}, Beckwith and Mertens used the modularity of $F_{R,r}-F_{R,R-r}$ to obtain more detailed asymptotics for $\mathcal{T}_{R,r}(n)-\mathcal{T}_{R,R-r}(n)$. Moreover, they obtained \cite{BM17} the leading contributions of $\mathcal{T}_{R,r}$ using Wright's treatment of the Circle Method, as well as the asymptotic Euler--Maclaurin summation formula (for more details, see Corollary~\ref{cor:asymptoticleading}). 
In particular, since the leading term, as $n\to \infty$, is independent of $r$, it follows that parts in partitions of $n$ asymptotically become equidistributed over all residue classes. 

In this paper, we give a more complete treatment of the asymptotic behavior of~$\mathcal{T}_{R,r}(n)$. We do so by studying the integral error term in the Euler--Maclaurin formula recalled in Section~\ref{sec:preliminaries} and recasting it as a Mordell-type integral to obtain the exact behavior under modular transformations in Theorem~\ref{thm:modulartransformation}. This makes it amenable for sharper asymptotic results (and an exact formula if the weight were appropriate). In particular, our main result is the following more refined result for the asymptotic behavior of $\mathcal{T}_{R,r}(n)$.

\begin{thm}\label{thm:asymptotic} Let $n,R \in \IN$, $24\nmid R$, and $r \in \{1,2,\ldots, R\}$. Then,\footnote{Here and throughout, we use the notation $\d_{\mathcal{S}} :=1$ if a statement $\mathcal{S}$ is true and we let $\d_{\mathcal{S}} :=0$ otherwise.} as $n \to \infty$ with $n_s:=n-\frac{1}{24}$, 
\begin{align*}
\calT_{R,r}(n) 
& =\frac{1}{(24n_s)^{\frac{1}{4}}} \vast(
\sum_{k=1}^{\lfloor \sqrt{n} \rfloor} \frac{\d_{\gcd(R,k) \mid r}}{\mathrm{lcm} (R,k)}
K^{[1]} (n,k) 
\lp \log\!\lp \frac{2\sqrt{6n_s}}{k} \rp I_{\frac{1}{2}} \!\lp \frac{\pi}{k} \sqrt{\frac{2n_s}{3}} \rp - \mathbb I_{\frac{1}{2}} \!\lp \frac{\pi}{k} \sqrt{\frac{2n_s}{3}} \rp \rp \\
&\quad+ {2\pi}
\sum_{k=1}^{\lfloor \sqrt{n} \rfloor} \frac{K_{R,r}^{[2]} (n,k)}{k^2}
I_{\frac{1}{2}}\!\lp \frac{\pi}{k} \sqrt{\frac{2n_s}{3}} \rp 
+\frac{\pi}{\sqrt{6n_s}} 
\sum_{k=1}^{\lfloor \sqrt{n} \rfloor} \frac{K_{R,r}^{[3]} (n,k)}{k}
I_{\frac{3}{2}} \!\lp \frac{\pi}{k} \sqrt{\frac{2n_s}{3}} \rp  \\
&\quad+2\pi
\sum_{\substack{\a_1,\a_2 \ge 1 \\ \a_1\a_2 < \frac{R}{24}}}  
\lp 1 - \frac{24\a_1\a_2}{R} \rp^{\!\frac{1}{4}}
\sum_{k=1}^{\lfloor \sqrt{n} \rfloor} \frac{K_{R,r,\bm{\a}}^{[4]} (n,k)}{k}
I_{\frac{1}{2}} \!\lp \frac{\pi}{k} \sqrt{\frac{2}{3}\lp 1 - \frac{24\a_1\a_2}{R} \rp n_s} \rp  \\
&\quad+\!\!\!
\sum_{\a_1,\a_2 =0}^{R-1} \sum_{k=1}^{\lfloor \sqrt{n} \rfloor}
\sum_{\k_1,\k_2=0}^{k-1}\!\!
\frac{K_{R,r,\bm{\a},\bm{\k}}^{[5]} (n,k)}{k} 
\operatorname{PV}\! \int_0^{\frac{1}{24}} 
\Phi_{R,k,R \bm{\k}+\bm{\a}}\! \left(\frac{1}{24}-t\right) \!(24t)^{\frac{1}{4}} 
I_{\frac{1}{2}}\! \lp \! \frac{4\pi}{k} \sqrt{n_st}  \rp \! dt\!\vast) \\
&\hspace{351pt} \!\!\!\!\!\! + O_R\!\lp n^{\frac{3}{4}} \log (n) \rp,
\end{align*}
where $K^{[1]} (n,k)$, $K_{R,r}^{[2]} (n,k)$, $K_{R,r}^{[3]} (n,k)$, $K_{R,r,\bm{\a}}^{[4]} (n,k)$, and $K_{R,r,\bm{\a},\bm{\k}}^{[5]} (n,k)$ are Kloosterman-like sums defined in \eqref{eq:Kloosterman5_definition}, \eqref{eq:Kloosterman4_definition}, \eqref{eq:Kloosterman2_3_definition}, and \eqref{eq:Kloosterman1_definition}, respectively, $I_\kappa$ is the Bessel function \eqref{eq:I_Bessel_Hankel_integral}, $\mathbb I_\kappa$ its derivative with respect to the order as in \eqref{eq:I_Bessel_derivative_c_infty_integral}, and $\Phi_{R,k,R \bm{\k}+\bm{\a}}$ an integral kernel defined in \eqref{eq:phi_function_definition}.
\end{thm}
\begin{remark}
We impose the condition $24 \nmid R$ to avoid complications in Lemma~\ref{lem:That_5_asymptotic_result} (and subsequently in Lemma~\ref{eq:leading_exponential_asymptotic_That_5}). This condition ensures that the kernel function $\Phi_{R,k,R \bm{\k}+\bm{\a}}$ does not have a pole on the boundary of the integration interval $[0,\frac{1}{24}]$.
It would be interesting to study the case $24 \mid R$. Such considerations would also be relevant in contexts involving false theta functions to generalize the methods used to obtain exact formulas found in \cite{BN19, Ces23}. 
\end{remark}

More concretely, the leading exponential asymptotic of $\calT_{R,r} (n)$ is given in the following corollary.
\begin{cor}\label{cor:leading_exponential_asymptotic}
Let $n,R \in \IN$, $24\nmid R$, and $r \in \{1,2,\ldots, R\}$. Then we have
\begin{align*}
\calT_{R,r} (n)  
&= 
\frac{e^{\pi \sqrt{\frac{2n_s}{3}}}}{4 \pi R \sqrt{2 n_s}}
\left( \phantom{2 \CE\!\lp 2 \pi \sqrt{\frac{2n_s}{3}} \rp} \hspace{-2.6cm}
\log\!\lp n_s\rp - \log \!\lp \frac{\pi^2}{6} \rp 
- 2 \psi\!\lp \frac{r}{R} \rp - 2 \log (R)
+ \frac{\pi(2r-R)}{2\sqrt{6n_s}}
+ \frac{R-2r}{4 n_s}\right.
\\ & \quad \quad \ \left.
+ 2 \CE\!\lp 2 \pi \sqrt{\frac{2n_s}{3}} \rp
+ 2 \int_0^{\e_R} \Psi_{R,r} (t) 
e^{- \pi \sqrt{\frac{2n_s}{3} } \left(1-\sqrt{1-24t}\right)} dt
\right)
+O_R \!\lp e^{\pi \sqrt{\frac{2}{3} (1- 24\e_R) n_s }} \rp \! ,
\end{align*}
where\footnote{
Here and throughout, we let $\z_N := e^{\frac{2\pi i}{N}}$ for $N \in \IN$.
}
\begin{equation}\label{eq:leading_exp_term_E_Psi_definitions}
\CE (x) := \mathrm{Ei}(x) e^{-x},
\ \ 
\e_R := \frac{1}{2} \mathrm{dist} \!\lp \frac{1}{24}, \frac{\IZ}{R} \rp,
\ \ \mbox{and} \ \ 
\Psi_{R,r} (t) := \sum_{\substack{n_1,n_2\in \Z\setminus \{0\}}} 
\frac{ R^2 t\zeta_R^{r n_1}}{n_1n_2(Rt-n_1n_2)}.
\end{equation}
Here $\psi(x)\!:=\!\frac{\Gamma'(x)}{\Gamma(x)}$ denotes the digamma function, $\mathrm{Ei} (x) := - \mathrm{PV} \int_{-x}^\infty \frac{e^{-t}}{t} dt$ for $x>0$ the exponential integral, and $\mathrm{dist} (x,S)$ the Euclidean distance between a point $x$ and a set $S$.
\end{cor}

From Corollary~\ref{cor:leading_exponential_asymptotic} we conclude the following asymptotics for $\mathcal{T}_{R,r}(n)$.
\begin{cor}\label{cor:asymptoticleading} 
Let $n,R \in \IN$, $24\nmid R$, and $r \in \{1,2,\ldots, R\}$. Then, for any $L \in \IN_0$, we have
\begin{equation}\label{eq:leading_exponential_asymptotic_power_suppressed}
\calT_{R,r} (n)  = 
\frac{e^{\pi \sqrt{\frac{2n_s}{3}}}}{4 \pi R \sqrt{2 n_s}}
\lp
\log \!\lp n_s\rp 
+
\sum_{\ell =0}^L \frac{a_{R,r,\ell}}{n_s^{\frac{\ell}{2}}}
+
O_{L,R}\!\lp n^{-\frac{L+1}{2}} \rp
\rp,
\end{equation}
where
\begin{align*}
a_{R,r,0} := - \log \!\lp \frac{\pi^2}{6} \rp 
- 2 \psi\!\lp \frac{r}{R} \rp - 2 \log (R),\quad 
a_{R,r,1} := \frac{6 + \pi^2(2r-R)}{2 \sqrt{6} \pi}.
\end{align*}
Moreover, for $\ell \geq 2$, we define
\begin{equation*}
a_{R,r,\ell} := \frac{R-2r}{4} \d_{\ell=2} + 2 (\ell-1)!
\lp \frac{\sqrt{3}}{2\sqrt2\pi} \rp^{\ell}  \lp 1 - \! (-1)^{\ell} \!
\sum_{j = \left\lceil \frac{\ell}{4} \right\rceil}^{\left\lfloor \frac{\ell}{2} \right\rfloor}
\frac{\ell}{2j} \! \lp \frac{2\pi^2 R}{3} \rp^{2j} \!\!
\frac{B_{2j} B_{2j} \!\lp \frac{r}{R} \rp}{(\ell-2j)! (4j-\ell)! (2j)!}
\rp\! .
\end{equation*}
Here, for $j\in\IN_0$, $B_j$ and $B_j(x)$ denote the Bernoulli numbers and polynomials, respectively.
\end{cor}

Such results can also be used to deduce the asymptotic behavior of $\calT_{R,r} (n) - \calT_{R,R-r} (n)$, whose generating function is a modular form.
\begin{cor}\label{cor:leading_exponential_asymptotic_antisymmetric}
Let $n,R \in \IN$, and $r \in \{1,2,\ldots, R-1\}$. Then we have
\begin{multline*}
\mathcal{T}_{R,r}(n)-\mathcal{T}_{R,R-r}(n)
=
\frac{e^{\pi \sqrt{\frac{2n_s}{3}}}}{2 \pi R \sqrt{2 n_s}}
\left(b_{R,r,0}+\frac{b_{R,r,1}}{\sqrt{n_s}}+\frac{b_{R,r,2}}{n_s}\right)
\\
+
\frac{1}{R}  \sqrt{\frac{2}{n_s}}
\sum_{\substack{\a_1,\a_2 \geq 1 \\ \a_1\a_2 < \frac{R}{32}}}
\sin \!\lp \frac{2 \pi  r\a_1}{R} \rp
e^{\pi \sqrt{\frac{2}{3} \! \lp 1 - \frac{24\a_1\a_2}{R} \rp n_s}} 
+O_R\!\lp e^{\pi \sqrt{\frac{n_s}{6}}} \rp ,
\end{multline*}
where 
\begin{equation*}
b_{R,r,0} := \pi \cot\!\lp \frac{\pi r}{R} \rp, \quad
b_{R,r,1} := \frac{\pi(2r-R)}{2\sqrt{6}}, \andd
b_{R,r,2} := \frac{R-2r}{4}.
\end{equation*}
\end{cor}

\begin{remark}
Corollary~\ref{cor:asymptoticleading} can also be obtained through a combination of the Euler--Maclaurin summation formula and Wright's work on the Circle Method. This was done by Beckwith--Mertens \cite[Theorem~1.3]{BM17}, who derived \eqref{eq:leading_exponential_asymptotic_power_suppressed} for $L=0$.
Note that two small corrections to the leading asymptotic \cite[Theorem~1.1]{BM15} are required. In our notation, that theorem states that for $R\geq 3$ and $1\leq r < \frac{R}{2}$ with $\gcd(r,R)=1$ there are explicit constants $c_{R,r,0}$ and $c_{R,r,1}$ such that $\mathcal{T}_{R,r}(n)-\mathcal{T}_{R,R-r}(n)$ equals
\begin{equation*}
\frac{e^{\pi \sqrt{\frac{2n_s}{3}}}}{\sqrt{n_s}}\left(c_{R,r,0}+\frac{c_{R,r,1}}{\sqrt{n_s}} \right)+ O_R\!\left(n^2e^{\pi \sqrt{\frac{n_s}{6}}}\right).
\end{equation*} Corollary~\ref{cor:leading_exponential_asymptotic_antisymmetric} adds the term involving $b_{R,r,2}$ along with the exponentially smaller terms on the second line, whose presence can be traced to \eqref{eq:leading_exponential_asymptotic_That_4} and which can not be ignored for $R > 32$. 
Applying these corrections to \cite[Example 1.2]{BM15}, we instead find
\begin{equation*}
\mathcal{T}_{3,1}(n)-\mathcal{T}_{3,2}(n) = \frac{e^{\pi \sqrt{\frac{2n_s}{3}}}}{6\sqrt{6n_s}}\left(1-\frac{1}{2 \sqrt{2n_s}}+\frac{\sqrt{3}}{4 \pi n_s}\right) 
+ O\!\left(e^{\pi \sqrt{\frac{n_s}{6}}}\right).
\end{equation*}
Moreover we have the following example, where the presence of the second line from Corollary \ref{cor:leading_exponential_asymptotic_antisymmetric} is clear (as can also be confirmed by numerical computations),
\begin{equation*}
\mathcal{T}_{48,1}(n)-\mathcal{T}_{48,47}(n) 
= \frac{e^{\pi \sqrt{\frac{2n_s}{3}}}}{96\sqrt{2n_s}}
\left(\cot \!\lp \frac{\pi}{48} \rp 
-\frac{23}{\sqrt{6n_s}}+\frac{23}{2 \pi n_s}\right) 
+ \frac{\sin \!\lp \frac{\pi}{24} \rp e^{\pi \sqrt{\frac{n_s}{3}}}}{24\sqrt{2n_s}}
+ O\!\left(e^{\pi \sqrt{\frac{n_s}{6}}}\right) .
\end{equation*}
\end{remark}

From Corollary~\ref{cor:asymptoticleading}, we obtain an asymptotic formula for the probability that a part $m$ in a partition of~$n$ lies in an arithmetic progression of the form $m\equiv r\pmod{R}$, where all the parts are counted with the same weight.

\begin{cor}\label{cor:probability}
If $n,R \in \IN$, $24\nmid R$, and $r \in \{1,2,\ldots, R\}$, then we have, as $n\to\infty$, 
\begin{multline*}
\frac{\mathcal{T}_{R,r}(n)}{\mathcal{T}_{1,1}(n)}= \frac{1}{R}\left(1 -\frac{\log (R)+\psi\!\left(\frac{r}{R}\right)+\gamma}{\log\! \left(\frac{e^\gamma\sqrt{6n_s}}{\pi }\right)}\right.\\
\left.+\frac{\pi}{4 \sqrt{6n_s} \log\!\left(\frac{e^\gamma\sqrt{6n_s}}{\pi }\right)}\left(2 r-R-1+\frac{\left(6+\pi ^2\right) \left(\log(R)+\psi\!\left(\frac{r}{R}\right)+\gamma\right)}{\pi^2\log\!\left(\frac{e^\gamma\sqrt{6n_s}}{\pi }\right)}\right)\!\right)+O_R\!\left(\frac{1}{n\log(n)}\right),
\end{multline*}
where $\gamma=-\psi(1)$ denotes the Euler--Mascheroni constant.
\end{cor}

Finally, we record a further application pertaining to hook lengths of partitions, which appear in various areas of combinatorics, number theory, and representation theory. For a partition, the number of standard Young tableaux is given in terms of the hook lengths by the Frame--Robinson--Thrall formula, which also describes the degree of the corresponding irreducible representation of the symmetric group. Hook lengths also arise in broader contexts, such as the Nekrasov--Okounkov identity and the determination of Siegel--Veech constants, which connects them to mathematical physics, modular forms, and enumerative geometry (see \cite[(6.12)]{NO} and \cite{CMZ18}, respectively).

For a partition $\lambda$, we let $H(\lambda) := \{h(\xi): \xi \in Y_\lambda\}$ be the set of hook lengths corresponding to the cells of the Young diagram $Y_\lambda$ of $\lambda$. We then define
\[
H_{R,r}(\lambda) := \sum_{\substack{h\in H(\lambda) \\ h\equiv r \pmod R}} \frac{1}{h},
\]
to be the sum of reciprocals of hook lengths of $\lambda$ in a given residue class $r\pmod R$ and we let 
\[
\mathcal{H}_{R,r}(\lambda) := \sum_{\lambda \vdash n} H_{R,r}(\lambda),
\]
be the sum of reciprocals of hook lengths congruent to $r \pmod{R}$ in all partitions of $n$. A key fact is that the multiset of hook lengths of all partitions of $n$ is
equal to the multiset obtained by taking the union over all partitions $\lambda$ of~$n$ of $\lambda_j$ repeated $\lambda_j$ times (see, e.g., \cite{BM01}). Hence, we have
\[  \mathcal{H}_{R,r}(n) = \mathcal{T}_{R,r}(n) \]
and the aforementioned results for parts of partitions in residue classes can be reinterpreted in terms of hook lengths. In particular, the asymptotic formula in Corollary~\ref{cor:probability} can be viewed as the probability that a hook length in a partition of~$n$ is $r \pmod{R}$, where the hook lengths~$h$ are counted with weight~$h^{-1}$. 
In a similar direction, Males showed that hook lengths asymptotically equidistribute in residue classes if each hook length is counted with the same weight \cite{Mal23}. Moreover, in~\cite{GOT24} the asymptotic distribution of hook lengths was obtained. 

The paper is organized as follows. After introducing the preliminary background in Section~\ref{sec:preliminaries}, in Section~\ref{sec:modular_transformations} we investigate $F_{R,r}$  near rationals and prove Theorem~\ref{thm:modulartransformation}. Next, in Section~\ref{sec:nonprincipal_bounds}, we bound the constants and nonprincipal parts following from that theorem and use these results to prove Theorem~\ref{thm:asymptotic} in Section~\ref{sec:circle_method}. Finally, in Section~\ref{sec:leading}, we obtain Corollaries~\ref{cor:leading_exponential_asymptotic}, \ref{cor:asymptoticleading}, and \ref{cor:probability}.

\section*{Acknowledgments}
The first author has received funding from the European Research Council (ERC) under the European Union’s Horizon 2020 research and innovation programme (grant agreement No. 101001179).
The second and third authors are supported by the SFB/TRR 191 “Symplectic Structure in Geometry, Algebra and Dynamics”, funded by the DFG (Projektnummer 281071066 TRR 191).

\section{Preliminaries}\label{sec:preliminaries}

\subsection{Bernoulli polynomials and the shifted Euler--Maclaurin formula}\label{sec:EM}
To develop the behavior of our generating functions near rational numbers, we employ a direct approach through the Euler--Maclaurin formula instead of relying on the decomposition into indefinite and false-indefinite theta functions or on the relation to the Maass--Eisenstein series. For this, we first recall Bernoulli numbers $B_\ell$ and Bernoulli polynomials $B_\ell (x)$ for $\ell \in \IN_0$ that satisfy
\begin{equation*}
\frac{t e^{tx}}{e^t-1} = \sum_{\ell\ge0} B_\ell (x) \frac{t^\ell}{\ell!}
\andd
B_\ell = B_\ell (0) .
\end{equation*}
We have, for $0 \leq x \leq 1$ if $\ell \geq 2$ and for $0<x<1$ if $\ell=1$ (see \cite[24.8.3]{NIST})
\begin{equation}\label{eq:Bernoulli_polynomial_identity}
B_\ell(x) = - \ell! \sum_{m\geq 1} \frac{2\cos\!\left(2\pi mx-\frac{\pi \ell}2\right)}{(2\pi m)^\ell}.
\end{equation}
Let $\wt{B}_{\ell}(x):=B_{\ell}(x - \lfloor x \rfloor)$.
A standard argument via integration by parts (see, e.g., Chapter 8 of~\cite{Olv}) leads to the shifted Euler--Maclaurin formula.

\begin{prop}[\textbf{Shifted Euler--Maclaurin formula}]\label{prop:euler_maclaurin}
Let $a \in [0,1]$ and $L\in\IN$. If $f$ is $L$ times continuously differentiable on $[M_1,M_2]$ where $M_1,M_2\in\mathbb{Z}$ with $M_2 > M_1$, then we have 
\begin{multline*}
\sum_{m=M_1}^{M_2-1} \!\! f(m+a)
=
\int_{M_1}^{M_2} f(x) dx
-
\sum_{m=0}^{L-1} \frac{B_{m+1} (a)}{(m+1)!}\lp f^{(m)}(M_1)-f^{(m)}(M_2) \rp
\\
+
\frac{(-1)^{L+1}}{L!} \int_{M_1}^{M_2} \wt{B}_L(x-a) f^{(L)}(x) dx.
\end{multline*}
\end{prop}

If $f$ is smooth and decays sufficiently rapidly towards infinity along with its derivatives, then we can let $M_2 \to \infty$ in Proposition~\ref{prop:euler_maclaurin}. If further $f$ is holomorphic, then we can deform the resulting paths of integration to obtain the following corollary, where $D_\theta := \{ r e^{i \alpha} :  r \geq 0 \mbox{ and } |\alpha| \leq \theta \}$.

\begin{cor}\label{cor:shifted_euler_maclaurin_complex}
Let $a \in [0,1]$ and $L \in \mathbb{N}$.
Assume that, for some $0 \leq \theta < \frac{\pi}{2}$,  $f$ is a holomorphic function on a domain $\mathcal{D}$ containing $D_\theta$.
Also suppose that $f^{(m)} (w) \ll |w|^{-1-\e}$ with $\e > 0$ for $m \in \{0,1,\ldots,L\}$ as $|w| \to \infty$ in $\mathcal{D}$. Then, for any nonzero $z \in D_\th$, we have
\begin{equation*}
\sum_{m\ge0} f((m+a)z)
=
\frac{1}{z} \int_{0}^{\infty} \!\! f(x) dx
-
\sum_{m=0}^{L-1} \! \frac{B_{m+1} (a)}{(m+1)!}  f^{(m)}(0) z^m
+
\frac{(-1)^{L+1} z^L}{L!} \int_{0}^\infty \! \wt{B}_L(x-a) f^{(L)}(xz) dx.
\end{equation*}
\end{cor}

\subsection{An integral appearing in Mordell-type representations of modular transforms}
The modular transformations of the false-indefinite component of our generating function can be represented as Mordell-type integrals, which in turn are crucial for the asymptotic analysis here. In this section, we describe the technical features of an integral that appears in this analysis.

\begin{lem}\label{lem:pv_integral_holomorphy_bound}
Let $w \in \IC$ with $w_1 > 0$ and $\a \in \IR \setminus \{0\}$. Then
$w\mapsto \mathrm{PV}\int_{0}^\infty \frac{x  e^{-2 \pi w x}}{\a (x-\a)}  dx$
is holomorphic. Moreover, we have the bound
\begin{equation*}
\left| \mathrm{PV}\int_{0}^\infty	\frac{x  e^{-2 \pi w x}}{\a (x-\a)}  dx  \right|
\leq \pi e^{-2\pi \a w_1} \d_{\a > 0} + \frac{1}{\a^2 w_1^2} .
\end{equation*}
\end{lem}
\begin{proof}
The result follows from Subsection 4.2 of \cite{BCN} (see in particular the proof of Theorem 4.9).
\end{proof}

\subsection{Bessel functions}
We employ the $I$-Bessel functions $I_\kappa$ of order $\kappa > 0$ and their derivatives with respect to order.
By 8.406.1 and 8.412.2 of \cite{GR}, for $x > 0$, we have 
\begin{equation}\label{eq:I_Bessel_Hankel_integral}
I_\k (x) = \frac{1}{2\pi i} \int_{\mathcal{C}} z^{-\k-1} e^{\frac{x}{2} \lp z + \frac{1}{z} \rp} dz,
\end{equation}
with $\mathcal{C}$ denoting Hankel's contour starting from $- \infty$ below the real line, winding around the origin, and then going to $- \infty$ above the real line.
Since $\k > 0$, the integration path can be deformed to $c + i \IR$ for any $c > 0$. The derivative of $I_\kappa$ with respect to the order is then given by
\begin{equation}\label{eq:I_Bessel_derivative_c_infty_integral}
\mathbb I_\k (x) 
:= \frac{\partial}{\partial\kappa}I_\kappa(x) 
= - \frac{1}{2\pi i} \int_{c-i\infty}^{c+i\infty} z^{-\k-1} \Log (z) e^{\frac{x}{2} \lp z + \frac{1}{z} \rp} dz .
\end{equation}
The $I$-Bessel functions of half integral order can be expressed in terms of elementary functions. For example, we have
\begin{equation}\label{eq:I_Bessel_half_integral_examples}
I_{\frac{1}{2}} (x) = \sqrt{\frac{2}{\pi x}} \sinh (x)
\andd
I_{\frac{3}{2}} (x)= \sqrt{\frac{2}{\pi x}} \lp \cosh (x) - \frac{\sinh (x)}{x} \rp.
\end{equation}
By 10.38.6 of \cite{NIST}, we have
\begin{equation}\label{eq:I_Bessel_order_der_1_2_example}
\mathbb{I}_{\frac{1}{2}}(x) = - \frac{1}{\sqrt{2\pi x}}
\lp\CE (2x) e^x  + \CF (2x) e^{-x} \rp,
\end{equation}
with $\CE (x)$ as in \eqref{eq:leading_exp_term_E_Psi_definitions} and where $\CF (x) :=\mathrm{E}_1 (x) e^{x} $ with $\mathrm{E}_1 (x) := \int_x^\infty \frac{e^{-t}}{t} dt$ for $x>0$ denoting the respective exponential integral. 
Note that by 6.12.1 and 6.12.2 of \cite{NIST} we have
\begin{equation}\label{eq:exponential_integral_term_asymptotics}
\CE (x) \sim \frac{1}{x} \lp 1 + \frac{1!}{x} + \frac{2!}{x^2} + \ldots \rp
\andd
\CF (x) \sim \frac{1}{x} \lp 1 - \frac{1!}{x} + \frac{2!}{x^2} - \ldots \rp
\quad 
\mbox{as } x \to \infty.
\end{equation}

\section{Behavior near Rationals}\label{sec:modular_transformations}
To determine the asymptotics of $\calT_{R,r}(n)$ we investigate $F_{R,r}$  near rationals and determine its behavior under modular transformations. More specifically,
for $h, k \in \IZ$ with $0\leq h<k$ and $\gcd (h,k) =1$, we study
\begin{equation}\label{eq:FNr_rewriting_euler_maclaurin_prep}
F_{R,r} \!\lp \frac{h}{k} +  \frac{iz}{k^2} \rp\! 
=
\sum_{\ell =0}^{k-1}
\sum_{m \geq 0}
\frac{1}{\z_k^{-h(R\ell+r)} e^{\frac{2 \pi R}{k}  \lp m + \frac{R\ell+r}{Rk} \rp z}\!-\!1}.
\end{equation}

\subsection{Technical preliminaries}
We first examine the sum over $m$ in~\eqref{eq:FNr_rewriting_euler_maclaurin_prep}. Here and throughout, 
\begin{equation*}
(\!(x)\!) :=
\begin{cases}
x - \lfloor x\rfloor - \frac12 & \text{if $x\in\R\setminus\Z$},\\
0 & \text{if $x\in\IZ$}.
\end{cases}
\end{equation*}

\begin{prop}\label{prop:euler_maclaurin_final_result}
For $k \in \IN$, $\nu \in \IZ$, $\a \in \IQ \cap (0,1]$ with $\a \neq 1$ if $k \nmid \nu$, and $t > 0$, we have 
\begin{align*}
&\sum_{m\geq0} \frac{1}{\zeta_{k}^{-\nu}e^{(m+\a)t}-1}
\\
&\ \ 
=- \frac{\d_{k \nmid \nu}}{t} \!
\lp\! \pi i \left(\!\!\left(\frac{\nu}{k}\right)\!\!\right)
+ \frac{1}{2} \log \!\lp 4 \sin^2 \!\lp \frac{\pi \nu}{k}  \rp \!\rp \!\rp
- \!\frac{\d_{k \mid \nu}}{t} \! \lp \log (t) + \psi (\a) \rp
+ \frac{1}{2} \!\lp\! \a - \frac{1}{2} \rp \!
\lp 1 -i \d_{k \nmid \nu} \cot \!\lp \frac{\pi \nu}{k} \rp\! \rp
\\
& \hspace{23pt}
+ \frac{1}{t} \!\!\sum_{\substack{{\bm n} \in \IZ^2 + \left(0, \frac{\nu}{k}\right) \\n_1,n_2 \neq 0}}
e^{-2 \pi i \a n_1}
\lp \frac{1}{ n_1n_2} \mathrm{PV}\int_{0}^\infty	\frac{x  e^{-\frac{4  \pi^2x}{t}}}{x-n_1n_2}  dx
+\frac{\pi i}{2} \lp \sgn (n_1) + \sgn (n_2) \rp  e^{-\frac{4  \pi^2 n_1 n_2}{t}}
\rp.
\end{align*}
\end{prop}

\noindent We prove this result through a number of technical lemmas.
\begin{lem}\label{lem:Euler_Maclaurin_first_case}
For $k \in \IN$, $\nu \in \IZ$ with $k \nmid \nu$, $\alpha\in\Q\cap(0,1)$, and $t > 0$, we have
\begin{equation*}
\sum_{m\geq0} \frac{1}{\zeta_{k}^{-\nu}e^{(m+\a)t}-1}
= - \left(\!\!\left(\frac{\nu}{k}\right)\!\!\right) \frac{\pi i}{t}
- \log \!\lp 4 \sin^2 \!\lp \frac{\pi \nu}{k}  \rp \!\rp \frac{1}{2t}
+  2\sum_{m\ge1}\int_0^\infty \frac{\cos\left(2\pi m\left(x-\a\right)\!\right)}{\zeta_{k}^{-\nu}e^{tx}-1}  dx.
\end{equation*}
\end{lem}
\begin{proof}
Letting $f_{k,\nu} (w) := \frac{1}{\z_k^{-\nu} e^w -1}$, we have, by Corollary~\ref{cor:shifted_euler_maclaurin_complex},
\begin{equation*}
\sum_{m\geq0} f_{k,\nu} ((m+\alpha)t)
 =  \frac{1}{t} \int_{0}^{\infty}\!\!\!\! f_{k,\nu}(x) dx- B_1\!\left(\alpha\right)\! f_{k,\nu}(0)
- B_2\!\left(\alpha\right) \! f_{k,\nu}'(0) \frac{t}{2}
-\frac{t^2}{2} \int_0^\infty\!\!\!\! \wt{B}_2\!\left(x-\alpha\right) f_{k,\nu}^{(2)} (tx) dx.
\end{equation*}
We first compute
\begin{equation*}
\int_{0}^{\infty}\! f_{k,\nu}(x)dx  
= -\pi i \left(\!\!\left(\frac{\nu}{k}\right)\!\!\right)-\frac{1}{2} \log \!\lp 4 \sin^2 \!\lp \frac{\pi \nu}{k}  \rp \!\rp.
\end{equation*}
For the second integral, we use \eqref{eq:Bernoulli_polynomial_identity} and interchange the order of integral and sum using its absolute convergence to get
\begin{equation*}
\int_0^\infty \widetilde{B}_2\!\left(x-\alpha\right) f_{k,\nu}^{(2)}(tx) dx 
= \frac1{\pi^2} \sum_{m\ge1} \frac1{m^2} \int_0^\infty 
\cos\!\left( 2\pi m \!\left(x-\alpha\right)\right) f_{k,\nu}^{(2)}(tx) dx.
\end{equation*}
Now we use integration by parts twice with the boundary terms canceling the terms with $B_1 (\a)$ and $B_2 (\a)$ above by \eqref{eq:Bernoulli_polynomial_identity} and obtain the lemma.
\end{proof}

We next extend Lemma~\ref{lem:Euler_Maclaurin_first_case} to the case $k \mid \nu$.
\begin{lem}\label{lem:Euler_Maclaurin_second_case}
For $t>0$ and $\a \in \IQ\cap(0,1]$ we have
\begin{equation*}
\sum_{m\geq 0}\frac{1}{e^{(m+\alpha)t}-1} =
- \frac{\log (t)}{t} - \frac{\psi (\a)}{t} + \frac{\d_{\a = 1}}{4}
+ 2 \sum_{m \geq 1}
\int_0^\infty \left(\frac1{e^{tx}-1}-\frac{1}{tx}\right) \cos\!\left(2\pi m\left(x-\alpha\right)\right) dx .
\end{equation*}
\end{lem}
\begin{proof}
We use Proposition~\ref{prop:euler_maclaurin} for $f_t(x):= \frac{1}{e^{tx}-1}-\frac{1}{tx}$ with $M_1 =0$, $M_2 = M \in \IN$, $a=\a$, and $L=2$. The lemma follows by integrating by parts twice as in Lemma~\ref{lem:Euler_Maclaurin_first_case} and 
letting $M \to \infty$ with
\begin{equation*}
\int_0^M f_t(x)dx = - \frac{\log (t)}{t} - \frac{\log (M)}{t} + o(1)
\ \ \mbox{and} \ \ 
\sum_{m=0}^{M-1} \frac{1}{m+\a} = \log (M) - \psi (\a) + o(1)
\quad \mbox{as } M \to \infty
\end{equation*}
for $\a \in \IQ$ with $0 < \a \leq 1$ (see Lehmer's discussion in \cite{Leh}).
\end{proof}

We next rewrite the integrals appearing in Lemmas~\ref{lem:Euler_Maclaurin_first_case} and~\ref{lem:Euler_Maclaurin_second_case} in terms of Mordell-type integrals.
\begin{lem}\label{lem:cos_shifted_integrals}
Let $t>0$.
\begin{enumerate}[label=\rm(\arabic*),leftmargin=*,ref=\rm(\arabic*)]
\item\label{lempart:cos_shifted_integral_with_nu} 
For $k \in \IN$, $\nu \in \IZ$ with $k \nmid \nu$ and $\alpha\in\Q\cap(0,1)$, we have 
\begin{multline*}
\hspace{.25cm}2\sum_{m\ge1}\int_0^\infty \frac{\cos\!\left(2\pi m \!\left(x-\a\right)\!\right)}{\zeta_{k}^{-\nu}e^{tx}-1}  dx
=
\frac{1}{2} \lp \a - \frac{1}{2} \rp \lp 1 -i \cot \!\lp \frac{\pi \nu}{k} \rp \rp
\\
+ \frac{1}{t} \!\sum_{\substack{{\bm n} \in \IZ^2  +\left(0,\frac{\nu}{k}\right)\\n_1\neq 0}} \!
e^{-2 \pi i \a n_1}
\lp \frac{1}{n_1n_2} \mathrm{PV}\int_{0}^\infty	\frac{x  e^{-\frac{4  \pi^2x}{t}}}{x\!-\!n_1n_2}  dx
+\frac{\pi i}{2} \lp \sgn (n_1) + \sgn (n_2) \rp  e^{-\frac{4  \pi^2 n_1 n_2}{t}}
\rp \!.
\end{multline*}

\item\label{lempart:cos_shifted_integral_without_nu} 
For $\alpha\in\Q\cap(0,1]$, we have
\begin{multline*}
\hspace{.2cm}2 \sum_{m \geq 1}
\int_0^\infty \left(\frac1{e^{tx}-1}-\frac{1}{tx}\right) \cos\!\left(2\pi m\left(x-\alpha\right)\right) dx
= \frac{1}{2} \lp \a - \frac{1}{2} \rp - \frac{\d_{\a=1}}{4}
\\
+ \frac{1}{t} \sum_{\substack{n_1,n_2\in\Z\setminus \{0\}}}
e^{-2 \pi i \a n_1}
\lp \frac{1}{n_1n_2} \mathrm{PV}\int_{0}^\infty	\frac{x  e^{-\frac{4  \pi^2x}{t}}}{x-n_1n_2}  dx
+\frac{\pi i}{2} \lp \sgn (n_1) + \sgn (n_2) \rp  e^{-\frac{4  \pi^2 n_1 n_2}{t}}
\rp .
\end{multline*}
\end{enumerate}
\end{lem}
\begin{proof}
\ref{lempart:cos_shifted_integral_with_nu} We first claim that
\begin{equation}\label{eq:lem_sin_pv_claim2}
\int_0^{\infty} \frac{\sin (2 \pi m x)}{\zeta_{k}^{-\nu}e^{tx}-1} dx
=
-\frac{1}{4 \pi m}+
\sum_{\ell\in \Z}\int_{0}^\infty
\frac{\left(tx-4\pi^2 i m \frac{\nu}{k}\right)\sin(x)}{\left(tx-4\pi^2 i m  \frac{\nu}{k}\right)^2+\left(4\pi^2 m\ell\right)^2} dx.
\end{equation}
Note that the poles of the integrands on the right-hand side are off the real line.
To prove \eqref{eq:lem_sin_pv_claim2}, we integrate by parts twice to find
\begin{multline*}
\int_{0}^\infty
\frac{\left(tx-4\pi^2 i m \frac{\nu}{k}\right)\sin(x)}{\left(tx-4\pi^2 i m  \frac{\nu}{k}\right)^2+\left(4\pi^2 m \ell\right)^2} dx
=
- \frac{i \nu}{4 m k} \frac{1}{\lp \frac{\pi i \nu}{k} \rp^2 + \pi^2 \ell^2}
- \frac{3 i \nu t^2}{32 \pi^2 m^3 k}
\frac{1}{\lp \lp \frac{\pi i \nu}{k} \rp^2 + \pi^2 \ell^2 \rp^2}
\\
+\frac{8 t^2}{(4 \pi m)^6}
\int_0^\infty
\frac{\lp tx-4\pi^2 i m\frac{\nu}{k} \rp \lp \left(3 t^2 - \left(t x-4\pi^2 i m\frac{\nu}{k} \right)^2\right)\sin (x)
- 3 t \left(t x-4\pi^2 i m\frac{\nu}{k} \right)\cos (x) \rp}{\lp   \left(\frac{tx}{4 \pi m}- \frac{\pi i \nu}{k} \right)^2 + \pi^2 \ell^2 \rp^3} dx  .
\end{multline*}
Note that the integrand on the right-hand side is
\smash{$\ll_{m,t,\frac{\nu}{k}} | \ell + \frac{1}{2} |^{-\frac{3}{2}} (x+1)^{-\frac{3}{2}}$}.
So this integral combined with the sum over $\ell \in \IZ$ is absolutely convergent and 
we can switch their order.
Then we can compute the sum over $\ell$ using the function
\begin{equation}\label{eq:g0_definition}
g_0(w) := \frac{\coth(w)}{w} = \sum_{\ell \in \IZ} \frac{1}{w^2+\pi^2 \ell^2},
\end{equation}
which is meromorphic with poles on $\pi i \IZ$, 
and its derivatives 
\begin{equation*}
g_1(w) := \sum_{\ell \in \IZ} \frac{1}{(w^2+\pi^2 \ell^2)^2} 
= - \frac{1}{2w} g_0'(w)
\ \mbox{ and } \ 
g_2(w) := \sum_{\ell \in \IZ} \frac{1}{(w^2+\pi^2 \ell^2)^3} 
= - \frac{1}{4w} g_1'(w).
\end{equation*}
We then revert back the integration by parts done above. Doing this once while noting that $w^2 g_1(w) \to 0$ as $w \to \infty$ yields
\begin{multline*}
\sum_{\ell \in \IZ} \int_{0}^\infty
\frac{\left(tx-4\pi^2 i m \frac{\nu}{k}\right)\sin(x)}{\left(tx-4\pi^2 i m  \frac{\nu}{k}\right)^2+\left(4\pi^2 m \ell\right)^2} dx
=
- \frac{i \nu}{4 m k} g_0 \!\lp \frac{\pi i \nu}{k} \rp
\\
+  \frac{2t}{(4 \pi m)^4} \int_0^\infty \lp tx-4\pi^2 i m\frac{\nu}{k} \rp
\left(- \left(tx-4\pi^2 i m\frac{\nu}{k}\right)\cos(x)+t\sin(x)\right)
 g_1 \!\lp \frac{tx}{4 \pi m}- \frac{\pi i \nu}{k} \rp dx.
\end{multline*}
The next application of integration by parts requires more care because $w g_0 (w) \to 1$ as $w \to \infty$. So we first replace the boundary at infinity by $\frac{\pi}{2} + 2 \pi M m$ with $M \in \IN$. Then integrating by parts, using 
\smash{$w g_0 (w) = 1 + \frac{2}{e^{2w}-1}$}, and finally letting $M \to \infty$ proves \eqref{eq:lem_sin_pv_claim2}.

We next change variables $x \mapsto \frac{4 \pi^2 x}{t}$ on the right-hand side of \eqref{eq:lem_sin_pv_claim2} and split the integral as
\begin{equation*}
\int_{0}^\infty
\frac{\left(tx-4\pi^2 i m \frac{\nu}{k}\right)\sin(x)}{\left(tx-4\pi^2 i m  \frac{\nu}{k}\right)^2+\left(4\pi^2 m \ell\right)^2} dx
=
\frac{1}{2t} \int_0^\infty \frac{\sin \!\lp \frac{4 \pi^2 x}{t} \rp}{x-im \lp \ell+ \frac{\nu}{k} \rp} dx
+
\frac{1}{2t} \int_0^\infty \frac{\sin \!\lp \frac{4 \pi^2 x}{t} \rp}{x+im \lp \ell- \frac{\nu}{k} \rp} dx.
\end{equation*}
Note that, as $|\ell| \to \infty$, we have
\begin{equation*}
\frac{1}{2t} \int_0^\infty \frac{\sin \!\lp \frac{4 \pi^2 x}{t} \rp}{x \mp im \lp \ell \pm \frac{\nu}{k} \rp} dx = \frac{\pm i}{8 \pi^2 m \lp \ell \pm \frac{\nu}{k} \rp} + O_{m,t,\frac{\nu}{k}}\!\left(\ell^{-2}\right).
\end{equation*}
So we can split the contribution of the two terms in the sum on $\ell$ 
if we take the sum symmetrically.\footnote{
Here and throughout, $\sum^*_{\ell \in S}$ means $\lim_{M \to \infty} \sum_{\ell \in S \cap [-M,M]}$.
}
Plugging this into \eqref{eq:lem_sin_pv_claim2}, then yields
\begin{equation*}
\int_0^{\infty} \frac{\sin (2 \pi m x)}{\zeta_{k}^{-\nu}e^{tx}-1} dx
=
-\frac{1}{4 \pi m}+
\frac{1}{t} \psum_{\ell\in \Z} \int_0^\infty \frac{\sin \!\lp \frac{4 \pi^2 x}{t} \rp}{x-im \lp \ell+ \frac{\nu}{k} \rp} dx.
\end{equation*}
Next, we note that by rotating the integration path by $\mp \frac{\pi}{2}$ and picking up the half residue contribution from the simple pole (if present) we have
(for $t>0$, $m \in \IN$, and $\l \in \IQ\setminus\{0\}$)
\begin{equation*}
\int_{0}^\infty	\frac{e^{\pm \frac{4  \pi ^2i x}{t}}}{x - i m \l} dx
= \mathrm{PV}\int_{0}^\infty	\frac{e^{-\frac{4  \pi ^2 x}{t}}}{x \mp m \l} dx + \frac{\pi i}{2} \lp  \sgn (\l) \pm 1 \rp e^{\mp \frac{4 \pi^2m \l}{t} } .
\end{equation*}
This then leads to 
\begin{multline*}
\int_0^\infty \frac{\sin (2\pi m x)}{\zeta_{k}^{-\nu}e^{tx}-1}  dx
=
- \frac{1}{4 \pi m} +
\frac{1}{2it} \ \psum_{\ell \in  \IZ + \frac{\nu}{k}} \left(
\mathrm{PV}\int_{0}^\infty	\frac{e^{-\frac{4  \pi^2x}{t}}}{x-m\ell}  dx
- \mathrm{PV}\int_{0}^\infty	\frac{e^{-\frac{4  \pi^2x}{t}}}{x+m\ell}  dx\right.
\\
\left.\vphantom{\frac{e^{-\frac{4  \pi^2x}{t}}}{x-m\ell}}+ \frac{\pi i}{2} (\sgn (\ell) +1) e^{-\frac{4  \pi^2 m \ell}{t}}
- \frac{\pi i}{2} (\sgn (\ell) -1) e^{\frac{4  \pi^2 m \ell}{t}}
\right).
\end{multline*}
The same arguments with $\sin (2\pi m x)$ replaced by $\cos (2\pi m x)$ yield
\begin{multline*}
\int_0^\infty \frac{\cos\!\left(2\pi m x\right)}{\zeta_{k}^{-\nu}e^{tx}-1}  dx
=
\frac{1}{2t} \sum_{\ell \in  \IZ + \frac{\nu}{k}} \left(
\mathrm{PV}\int_{0}^\infty	\frac{e^{-\frac{4  \pi^2x}{t}}}{x-m\ell}  dx
+ \mathrm{PV}\int_{0}^\infty	\frac{e^{-\frac{4  \pi^2x}{t}}}{x+m\ell}  dx \right.
\\[-.75em]
\left. \vphantom{\frac{e^{-\frac{4  \pi^2x}{t}}}{x-m\ell}}+ \frac{\pi i}{2} (\sgn (\ell) +1) e^{-\frac{4  \pi^2 m \ell}{t}}
+ \frac{\pi i}{2} (\sgn (\ell) -1) e^{\frac{4  \pi^2 m \ell}{t}}
\right).
\end{multline*}
The claim then follows by combining these two results, splitting
\begin{equation*}
\frac{1}{x-m\ell} = \frac{x}{m\ell (x-m\ell)} - \frac{1}{m\ell}
\andd
\frac{1}{x+m\ell} = - \frac{x}{m\ell (x+m\ell)} + \frac{1}{m\ell},
\end{equation*}
while noting the bound in Lemma~\ref{lem:pv_integral_holomorphy_bound} along with \eqref{eq:Bernoulli_polynomial_identity} and \eqref{eq:g0_definition}.

\noindent\ref{lempart:cos_shifted_integral_without_nu} is proved in a similar fashion by establishing first the identities (for $m \in \IN$)
\begin{align*}
\int_0^\infty \left(\frac1{e^{tx}-1}-\frac1{tx}\right) \cos(2\pi mx) dx &= \frac2t \sum_{\ell\ge1} \operatorname{PV} \int_0^\infty \frac{ue^{-\frac{4\pi^2u}t}}{u^2-(m\ell)^2} du,\\
\int_0^\infty \left(\frac1{e^{tx}-1}-\frac1{tx}\right) \sin(2\pi mx) dx &=
- \frac{1}{4 \pi m} + \frac{\pi}{t} \frac{1}{e^{\frac{4 \pi^2 m}{t}} - 1}.
\qedhere
\end{align*}
\end{proof}

The main result of this section now follows directly.
\begin{proof}[Proof of Proposition~\ref{prop:euler_maclaurin_final_result}]
The proposition follows by Lemmas~\ref{lem:Euler_Maclaurin_first_case}, \ref{lem:Euler_Maclaurin_second_case}, and~\ref{lem:cos_shifted_integrals}.
\end{proof}

\subsection{Modular transformation}\label{sec:modular_transformation_explicit}
Our next goal is to determine how the functions $F_{R,r}$ behave near the rational number $\frac{h}{k}$ by examining \eqref{eq:FNr_rewriting_euler_maclaurin_prep}. In other words, we study the behavior of $F_{R,r}$ under the transformation
\begin{equation}\label{eq:M_hk_definition}
M_{h,k} := \mat{h & - \frac{hh'+1}{k} \\ k & -h'} \in \mathrm{SL}_2(\IZ),
\end{equation}
where here and throughout $h' \in \IZ$ with $0 \leq h' < k$ and $hh' \equiv -1 \pmod{k}$. This is the modular transformation that maps $\frac{h'}{k} + \frac{i}{z}$ to $\frac{h}{k} + \frac{iz}{k^2}$. 
We define, for $\bm{\a} \in \IZ^2$,
\begin{align}
A_{R,r,h,k} &:= -\frac{k}{2 \pi R} \sum_{\ell =0}^{k-1} \left(
\pi i \left(\!\!\left(\frac{h (R \ell +r)}{k}\right)\!\!\right)
+ \frac{\d_{k \nmid (R \ell + r)}}{2}
\log \!\lp 4 \sin^2 \!\lp \frac{\pi h (R \ell +r)}{k}  \rp \!\rp\right. \notag
\\ & \left.\hspace{5.5cm}
+ \d_{k \mid (R \ell + r)} \lp \log \!\lp \frac{2 \pi R}{k} \rp +
\psi \!\lp \frac{R \ell + r}{Rk} \rp \rp
\right),
\label{eq:alpha_definition}
\\
B_{R,r,h,k} &:= \frac{1}{2} \sum_{\ell =0}^{k-1}
\lp \frac{R \ell +r}{Rk} - \frac{1}{2} \rp
\lp 1 - i \d_{k \nmid (R \ell + r)} \cot\!\lp \frac{\pi h (R\ell+r)}{k} \rp \rp ,
\label{eq:beta_definition}
\\
C_{R,r,k} &:= \frac{k^2 \d_{\gcd(R,k) \mid r}}{2 \pi \mathrm{lcm} (R,k)},
\label{eq:gamma_definition} \\
\psi_{R,r,h,k} ({\bm \a}) &:= \frac{\zeta_{Rk}^{-h'\alpha_1\alpha_2}}{Rk}\! \sum_{\nu_1, \nu_2 \pmod{k}}\!
e^{\frac{2\pi i}{k} \lp R h \nu_1 \left(\nu_2 + \frac{r}{R}\right) - \a_2 \nu_1 - \a_1 \left(\nu_2 + \frac{r}{R}\right)\rp}.
\label{eq:psi_definition}
\end{align}

\begin{remark}
Recall from \eqref{eq:FRr_indef_falseindef_decomposition} that $F_{R,r}$ is a linear combination of an indefinite and a false-indefinite theta function (up to corrections on the boundary of cones defining these functions).
These theta functions are over a lattice of signature~$(1,1)$ with Gram matrix $A := \pmat{0 & R \\ R & 0}$ (denoting the associated quadratic form by $Q(\bm{n}):=\frac{1}{2} \bm{n}^T A \bm{n}$ and the bilinear form by $B(\bm{n},\bm{m}):= \bm{n}^T A \bm{m}$).
For any two elements of the discriminant group $\bm{\mu}, \bm{\nu} \in A^{-1} \IZ^2 / \IZ^2$, the corresponding \emph{Weil multiplier} is given, for $M := \pmat{a & b \\ c & d} \in \mathrm{SL}_2(\IZ)$, (see, e.g.,~\cite[Section~14]{CS} for further details)
\begin{equation*}
\psi_M(\bm\mu,\bm\nu) :=
\begin{cases}
e^{2 \pi i ab Q(\bm{\mu})} \delta_{\bm{\mu}, \sgn (d) \bm{\nu}} \quad &\mbox{if } c=0,\\[3pt]
\frac{1}{|c|  \sqrt{|\det(A)|}} 
\displaystyle\sum_{\bm{m} \in \Z^2 / c\Z^2}
e^{\frac{2\pi i}{c} \lp a Q(\bm{m}+\bm{\mu}) - B(\bm{m}+\bm{\mu}, \bm{\nu}) + d  Q(\bm{\nu}) \rp}
\quad &\mbox{if } c\neq 0.
\end{cases}
\end{equation*}
Note that $\psi_{R,r,h,k} ({\bm \a})$ in \eqref{eq:psi_definition} is the Weil multiplier $\psi_{M_{h,k}} ( (0, \frac{r}{R}),\frac{{\bm \a}}{R})$ for the modular transformation $M_{h,k}$ from \eqref{eq:M_hk_definition}.
\end{remark}

We are now ready to show Theorem~\ref{thm:modulartransformation}.
\begin{proof}[Proof of Theorem~\ref{thm:modulartransformation}]
By Lemma~\ref{lem:pv_integral_holomorphy_bound}, the summands in \eqref{eq:calI_definition} are holomorphic and they can be bounded as $\ll_{K,R} n_1^{-2} n_2^{-2}$ uniformly for $\t$ in compact subsets $K$ of $\IH$. 
So $\mathcal{I}_{R,{\bm \a},-\frac{d}{c}}$ is a holomorphic function on $\IH$. In particular, this implies that both sides of the claimed identity are holomorphic in $z$ for $\re (z) > 0$. Thus it is sufficient to prove the identity for $z=t>0$.
We go back to \eqref{eq:FNr_rewriting_euler_maclaurin_prep} and use Proposition~\ref{prop:euler_maclaurin_final_result} with $\nu = h(R\ell+r)$, $\a = \frac{R \ell+r}{Rk}$, and $t \mapsto \frac{2 \pi R}{k} t$ to find 
\begin{align*}
&F_{R,r} \!\lp \frac{h}{k} + \frac{it}{k^2} \rp
=
- C_{R,r,k} \frac{\log (t)}{t}
+ \frac{A_{R,r,h,k}}{t} + B_{R,r,h,k}  \\
&\!+ \frac{k}{2\pi Rt}
\sum_{\ell =0}^{k-1}
\sum_{\substack{{\bm n} \in \IZ^2 + \left(0, \frac{h (R\ell +r)}{k}\right) \\ n_1,n_2 \neq 0}} \hspace{-0.4cm}
e^{-\frac{2\pi i(R\ell+r)n_1}{Rk}}\!\!
\left( \!\frac{1}{n_1n_2} \mathrm{PV}\! \! \int_{0}^\infty \! \frac{x  e^{-\frac{2 \pi k x}{Rt}}}{x\!-\!n_1n_2}  dx 
+ \frac{\pi i}{2} ( \sgn (n_1) \!+\! \sgn (n_2) )
e^{-\frac{2 \pi k n_1 n_2}{Rt}}\!
\right)\!.
\end{align*}
The first three terms on the right-hand side match the first three terms on the right-hand side of the claimed modular transformation. Our remaining task is to verify that the term involving~$f$ and~$\mathcal{I}$ equals the final term above, which we write as
\begin{equation}\label{eq:details_multiplier_modular_transf_gdefn2}
\frac{ik}{Rt}
\sum_{\ell =0}^{k-1}
\sum_{\substack{{\bm n} \in \IZ^2 + \left(0, \frac{h (R\ell +r)}{k}\right) \\ n_1,n_2 \neq 0}}
e^{-\frac{2\pi i(R\ell+r)n_1}{Rk}} g({\bm n}),
\end{equation}
where (changing $x \mapsto \frac{R}{k} x$ in the integral)
\begin{equation}\label{eq:details_multiplier_modular_transf_gdefn}
g({\bm n}) := \frac{1}{2 \pi i} \mathrm{PV}\int_{0}^\infty
\frac{x  e^{-\frac{2\pi x}t}}{\frac{k}{R} n_1 n_2 \left(x- \frac{k}{R} n_1n_2\right)}  dx
+\frac{1}{4} \lp \sgn (n_1) + \sgn (n_2) \rp  e^{-\frac{2 \pi k n_1 n_2}{Rt}}.
\end{equation}
Note that the sum in \eqref{eq:details_multiplier_modular_transf_gdefn2} is absolutely convergent by Lemma~\ref{lem:pv_integral_holomorphy_bound}. So we may rewrite \eqref{eq:details_multiplier_modular_transf_gdefn2} as
\begin{align}\label{eq:details_multiplier_modular_transf_with_multiplier_explicit}
&\hspace{-.2cm}\frac{i}{Rt}  \sum_{\nu_1,\nu_2 \pmod{k}}
\sum_{\substack{n_1,n_2\in\Z\setminus\{0\}}}
\zeta_{Rk}^{-(R\nu_2 +r)n_1} \zeta_{k}^{\nu_1\left(h\left(R\nu_2+r\right)-n_2\right)}
g\!\lp n_1,\frac{n_2}{k} \rp
\nonumber\\
&\hspace{.3cm}=\frac{i}{Rt} \sum_{\a_1,\a_2 \pmod{R}}
\sum_{\substack{{\bm n} \in \IZ^2 + \frac{\bm{\a}}{R} \\ n_1,n_2 \neq 0}}
\sum_{\nu_1,\nu_2 \pmod{k}}
e^{\frac{2\pi i}{k} \lp h R \nu_1 \lp \nu_2 +\frac{r}{R} \rp  - Rn_1 \lp \nu_2 +\frac{r}{R} \rp- R n_2 \nu_1- Rh' n_1n_2 \rp}
\nonumber\\[-1.5em]
&\hspace{10cm}\times e^{\frac{2\pi iRh'n_1n_2}{k}} g\!\lp Rn_1,\frac{Rn_2}{k} \rp  ,
\end{align}
changing ${\bm n} \mapsto R{\bm n}$ with ${\bm n} \in \IZ^2$ and $\a_1,\a_2$ running $\!\pmod R$.
The sum on $\nu_1,\nu_2$ depends on 
$\bm{n} \in \IZ^2 +  \frac{\bm{\a}}{R}$ only $\!\!\pmod1$, so it can be written as $Rk \psi_{R,r,h,k} ({\bm \a})$.
Therefore, \eqref{eq:details_multiplier_modular_transf_with_multiplier_explicit} equals
\begin{equation}\label{Verweis}
\frac{ik}{t}  \sum_{\a_1,\a_2 \pmod{R}}  \psi_{R,r,h,k} ({\bm \a})
\sum_{\substack{{\bm n} \in \IZ^2 + \frac{{\bm \a}}{R} \\ n_1,n_2 \neq 0}}
e^{\frac{2\pi iRh'n_1n_2}{k}} g\!\lp Rn_1,\frac{Rn_2}{k} \rp .
\end{equation}
Using \eqref{eq:f_indef_definition}, \eqref{eq:calI_definition}, and \eqref{eq:details_multiplier_modular_transf_gdefn}, the sum over ${\bm n}$ yields 
\smash{$f_{R,{\bm \a}}  (\frac{h'}{k} + \frac{i}{t}) + \mathcal{I}_{R,{\bm \a}\frac{h'}{k}} (\frac{h'}{k} + \frac{i}{t})$}
and \eqref{Verweis} matches the second line of Theorem~\ref{thm:modulartransformation}.
\end{proof}

\section{Bounds on Components of the Transformation}\label{sec:nonprincipal_bounds}
In this section, we bound the constants appearing in Theorem~\ref{thm:modulartransformation} along with nonprincipal parts.
We start by estimating $A_{R,r,h,k}$ and $B_{R,r,h,k}$.

\begin{lem}\label{lem:alpha_beta_bounds}
Let $R$, $k$, $r\in \IN$, $1\le r \le R$, and $h \in \IZ$ with $\gcd (h,k)=1$. Then we have
\begin{equation*}
A_{R,r,h,k} \ll_R k^2 \log(k) + \d_{k=1} \qquad \text{and} \qquad
B_{R,r,h,k} \ll k^2.
\end{equation*}
\end{lem}
\begin{proof}
The result follows from \eqref{eq:alpha_definition} and~\eqref{eq:beta_definition} while noting that the digamma function $\psi$ is continuous, strictly increasing, negative in $(0,1]$, and we have $\psi(x)=-\frac{1}{x}+O(1)$ as $x\to 0^+$.
\end{proof}

We continue by bounding the nonprincipal parts within the indefinite theta portion $e^{-2\pi i d \t} f_{R,{\bm \a}} (\t)$ mixed with an exponentially growing piece.

\begin{lem}\label{lem:indef_part_nonprincipal}
Let $0 \leq d \leq 1$, $R$, $k \in \IN$, $h' \in \IZ$ with $\gcd (h',k)=1$, and $\a_1, \a_2 \in \{0,1,\ldots,R-1\}$. Then, for $w \in \IC$ with $w_1 \geq 1$, we have
(with the error term bounded by an absolute constant) 
\begin{multline*}
e^{-2 \pi i d \left(\frac{h'}{k} + iw\right)} f_{R,{\bm \a}} \!\lp \frac{h'}{k} + iw \rp
= \frac{\d_{0 < \a_1\a_2 < Rd}}{2}  e^{2 \pi \left(d-\frac{\a_1 \a_2}{R}\right)\left(w-\frac{ih'}{k}\right)}
\\
- \frac{\d_{(R-\a_1)(R-\a_2) < Rd}}{2}
e^{2 \pi \left(d-\frac{(R-\a_1)(R-\a_2)}{R}\right)\left(w-\frac{ih'}{k}\right)}
+ O(1) .
\end{multline*}
\end{lem}
\begin{proof}
We start by defining
\begin{equation*}
J_{R,{\bm \a},d} (\t) \!:=\!e^{-2\pi i d \t} f_{R,{\bm \a}} (\t) 
\!-\! \!\lp\! \frac{\d_{0 < \a_1 \a_2 < Rd}}{2}  e^{-2 \pi i \lp d-\frac{{\a_1 \a_2}}{R} \rp \t}
\!-\! \frac{\d_{(R-\a_1)(R-\a_2) < Rd}}{2}  e^{-2 \pi i \lp d-\frac{(R-\a_1)(R-\a_2)}{R} \rp \t} \!\rp \! 
\end{equation*}
so that the lemma is equivalent to $J_{R,\bm{\a},d} (\t) = O(1)$  for $\t \in \IH$ with $\t_2 \geq 1$.
Using the definition of $f_{R,\bm{\a}}$ in \eqref{eq:f_indef_definition}, letting ${a_1} := \a_1 + R \d_{\a_1=0}$ and ${a_2} := \a_2 + R \d_{\a_2=0}$ so that ${a_1},{a_2} \in \{1,2,\ldots,R\}$, and splitting off the contribution from $\bm{n}=\bm0$ below, we find that
\begin{multline*}
|J_{R,\bm{\a},d} (\t)| \leq
\frac{\d_{{a_1}{a_2} \geq Rd}}{2} e^{- 2 \pi \lp \frac{{a_1}{a_2}}{R}  - d \rp \t_2}
+ \frac{\d_{(R-\a_1)(R-\a_2) \geq Rd}}{2} 
e^{- 2 \pi \lp \frac{1}{R} (R-\a_1)(R-\a_2) - d \rp \t_2}
\\
+\frac{1}{2} \sum_{\substack{{\bm n} \in \mathbb{N}_0^2\setminus \{\bm 0\} }} 
e^{- 2 \pi \lp R \lp n_1 +\frac{{a_1}}{R} \rp \lp n_2 + \frac{{a_2}}{R} \rp -d \rp \t_2}
+ \frac{1}{2}  \sum_{\substack{{\bm n} \in \mathbb{N}_0^2 \setminus \{\bm 0\}}} 
e^{- 2 \pi \lp R \lp n_1 +\frac{R-\a_1}{R} \rp \lp n_2 + \frac{R-\a_2}{R} \rp -d \rp \t_2}.
\end{multline*}
The claim now follows, using that $d\le 1$, the condition on $\a_1,\a_2$, and splitting off the contribution from $n_1=0$ and $n_2=0$.
\end{proof}

We continue with the Mordell-type components \eqref{eq:calI_definition} arising from the false-indefinite part mixed with an exponentially growing piece. More specifically, assuming $d \geq 0$ and $Rd \not\in \IN$, we decompose \eqref{eq:calI_definition} into principal and nonprincipal parts as
\begin{equation}\label{eq:calI_decomposition_principal_nonprincipal}
e^{2 \pi d w} \mathcal{I}_{R,{\bm \a},\frac{h'}{k}} \!\lp \frac{h'}{k} + iw \rp 
= 
\mathcal{I}^*_{R,{\bm \a},\frac{h'}{k},d} \!\lp \frac{h'}{k} + iw \rp 
+
\mathcal{I}^e_{R,{\bm \a},\frac{h'}{k},d} \!\lp \frac{h'}{k} + iw \rp ,
\end{equation}
where
\begin{align}
\mathcal{I}^*_{R,{\bm \a},\frac{h'}{k},d} \!\lp \frac{h'}{k} + iw \rp 
&:=
\frac{e^{2 \pi d w}}{2\pi i} 
\sum_{\substack{{\bm n}\in \mathbb{Z}^2 +\frac{{\bm \a}}{R}\\n_1,n_2 \neq 0 }}
\frac{e^{\frac{2\pi iRh'n_1n_2}{k}}}{Rn_1 n_2} 
\mathrm{PV}\!\int_0^d \!
\frac{x  e^{- 2 \pi w x}}{x\!-\!Rn_1n_2}  dx,
\label{eq:false_indef_part_principal_defn}
\\
\mathcal{I}^e_{R,{\bm \a},\frac{h'}{k},d} \!\lp \frac{h'}{k} + iw \rp 
&:=
\frac{e^{2 \pi d w}}{2\pi i} 
\sum_{\substack{{\bm n}\in \mathbb{Z}^2 +\frac{{\bm \a}}{R}\\n_1,n_2 \neq 0}}
\frac{e^{\frac{2\pi iRh'n_1n_2}{k}}}{Rn_1 n_2} 
\mathrm{PV}\!\int_d^\infty\!
\frac{x  e^{- 2 \pi w x}}{x\!-\!Rn_1n_2}  dx .
\label{eq:false_indef_part_nonprincipal_defn}
\end{align}
We next bound the nonprincipal part $\mathcal{I}^e$ while also proving the convergence of the given sum along with the holomorphy of the resulting function (which also proves the same for $\mathcal{I}^*$ as we show this for $\mathcal{I}$ in the proof of Theorem~\ref{thm:modulartransformation} in Subsection \ref{sec:modular_transformation_explicit}).

\begin{lem}\label{lem:false_indef_part_nonprincipal}
Let $R,k \in \IN$, $d \geq 0$ with $Rd \not\in \IN$, $h' \in \IZ$ with $\gcd (h',k)=1$, and ${\bm \a} \in \IZ^2$. Then, for $w \in \IC$ with $w_1 > 0$, \eqref{eq:false_indef_part_nonprincipal_defn} converges to a holomorphic function of $w$. For $w_1 \geq 1$, we have 
\begin{equation}\label{Mathcal}
\mathcal{I}^e_{R,{\bm \a}, \frac{h'}{k},d} \!\lp \frac{h'}{k} + iw \rp = O_{R,d} (1).
\end{equation}
\end{lem}
\begin{proof}
We follow the proof of Lemma 5.1 from \cite{BCN}, start with the change of variables $x \mapsto x+d$, and use $x+d = \frac{Rn_1n_2}{Rn_1n_2-d} x - \frac{d(x - (Rn_1n_2-d))}{Rn_1n_2-d} $
to rewrite the integral in \eqref{eq:false_indef_part_nonprincipal_defn} as
\begin{equation}\label{eq:false_indef_part_nonprincipal_decomposition}
e^{2 \pi d w} 
\mathrm{PV}\!\int_d^\infty\!
\frac{x  e^{- 2 \pi w x}}{x\!-\!Rn_1n_2}  dx
=
\frac{Rn_1n_2}{Rn_1n_2-d} 
\mathrm{PV}\!\int_0^\infty \frac{x e^{- 2 \pi w x}}{x- (Rn_1n_2-d)}  dx
- \frac{d}{Rn_1n_2-d} \frac{1}{2 \pi w},
\end{equation}
while noting that $Rn_1n_2-d \neq 0$.
The integral on the right-hand side converges to a holomorphic function of $w$ for $w_1 > 0$ by Lemma~\ref{lem:pv_integral_holomorphy_bound}.
We next show that the sum over ${\bm n}$ in \eqref{eq:false_indef_part_nonprincipal_defn} is absolutely and uniformly convergent over $w_1 \geq \d$ for any $\d > 0$, which then proves its holomorphy for $w_1 > 0$. We again use \eqref{eq:false_indef_part_nonprincipal_decomposition} and Lemma~\ref{lem:pv_integral_holomorphy_bound} to find
\begin{equation*}
\sum_{\substack{\bm{n} \in \IZ^2 + \frac{\bm{\a}}{R} \\ n_1,n_2 \neq 0}}
\left| \frac{e^{2 \pi d w}}{2\pi i} 
\frac{e^{\frac{2\pi iRh'n_1n_2}{k}}}{Rn_1 n_2} 
\mathrm{PV}\!\int_d^\infty\!
\frac{x  e^{- 2 \pi w x}}{x\!-\!Rn_1n_2}  dx \right|
\ll_{R,d,\d} 1 .
\end{equation*}
This completes the proof of holomorphy for $w_1 > 0$. The same bound with $\d=1$ yields \eqref{Mathcal}.
\end{proof}

\section{Proof of Theorem~\ref{thm:asymptotic}}\label{sec:circle_method}
In this section, we apply the Circle Method to \eqref{eq:That_generating_function} and prove Theorem~\ref{thm:asymptotic}. 
We use Rademacher's integration path for $N \in \IN$ (see, e.g.~\cite{Apostol} for details) to write (recalling $n_s = n-\frac{1}{24}$)
\begin{equation}\label{eq:That_Cauchy_expression}
\calT_{R,r}(n) 
= i \sum_{k=1}^N k^{-2} \sum_{\substack{0 \leq h < k \\ \gcd(h,k) = 1}} 
\int_{z_1}^{z_2} \frac{F_{R,r} \!\lp \frac{h}{k} + \frac{iz}{k^2} \rp}{\eta \!\lp \frac{h}{k} + \frac{iz}{k^2} \rp} 
e^{-2\pi i n_s \lp \frac{h}{k} + \frac{iz}{k^2} \rp} d z,
\end{equation}
where\footnote{Note that $z_1$ and $z_2$ should not be confused with the real and imaginary part of $z$.} $z_j := z_j (h,k;N)$ is
(with $\frac{h_1}{k_1} < \frac{h}{k} < \frac{h_2}{k_2}$ consecutive terms in the Farey sequence of order $N$)
\begin{equation*}
z_1 (h,k;N) := \frac{k}{k-ik_1} \andd
z_2 (h,k;N) := \frac{k}{k+ik_2} .
\end{equation*}
Here we take $k_1=k_2=N$ for $(h,k)=(0,1)$ and assume that $N = \lfloor \sqrt{n} \rfloor$ from now on.

Next, we use the transformation law of $F_{R,r}$ from Theorem~\ref{thm:modulartransformation} along with the modular properties of the Dedekind eta function (with $\nu_\eta$ denoting its multiplier system)
\begin{equation*}
\eta \!\lp \frac{h}{k} + \frac{iz}{k^2} \rp = \nu_\eta (M_{h,k}) \sqrt{\frac{ik}{z}}
\eta \!\lp \frac{h'}{k} + \frac{i}{z} \rp,
\end{equation*}
under the modular transformation $M_{h,k}$ from \eqref{eq:M_hk_definition}.
Defining (see \eqref{eq:psi_definition} for the Weil multiplier~$\psi_{R,r,h,k}$)
\begin{equation}\label{eq:nu_varphi_gamma_definition}
\nu_{h,k} := e^{-\frac{\pi i}{4}} \nu_\eta^{-1} (M_{h,k})
\andd
\varphi_{R,r,h,k} ({\bm \a}) := \nu_{h,k} \psi_{R,r,h,k} ({\bm \a}),
\end{equation}
we then split \eqref{eq:That_Cauchy_expression} into five pieces coming from the five terms appearing in Theorem~\ref{thm:modulartransformation}:
\begin{equation}\label{splitT}
\calT_{R,r}(n) =\sum_{j=1}^{5} \calT_{R,r}^{[j]}(n),
\end{equation}
where 
\begin{align}
\notag
\calT^{[1]}_{R,r}(n) &:=
- i \sum_{k=1}^N \frac{C_{R,r,k}}{k^{\frac{5}{2}} } 
\sum_{\substack{0 \leq h < k \\ \gcd(h,k) = 1}} \nu_{h,k}
\int_{z_1}^{z_2}
\frac{\Log (z)}{\sqrt{z} \eta \!\lp \frac{h'}{k} + \frac{i}{z} \rp}
e^{-2\pi i  n_s \lp \frac{h}{k} + \frac{iz}{k^2} \rp} d z,
\\ 
\notag
\calT^{[2]}_{R,r}(n) &:=
i \sum_{k=1}^N k^{-\frac{5}{2}} 
\sum_{\substack{0 \leq h < k \\ \gcd(h,k) = 1}} \nu_{h,k} A_{R,r,h,k} 
\int_{z_1}^{z_2}
\frac{1}{\sqrt{z} \eta \!\lp \frac{h'}{k} + \frac{i}{z} \rp}
e^{-2\pi i  n_s \lp \frac{h}{k} + \frac{iz}{k^2} \rp} d z,
\\ 
\notag
\calT^{[3]}_{R,r}(n) &:=
i \sum_{k=1}^N k^{-\frac{5}{2}} 
\sum_{\substack{0 \leq h < k \\ \gcd(h,k) = 1}} \nu_{h,k} B_{R,r,h,k} 
\int_{z_1}^{z_2}
\frac{\sqrt{z}}{\eta \!\lp \frac{h'}{k} + \frac{i}{z} \rp}
e^{-2\pi i  n_s \lp \frac{h}{k} + \frac{iz}{k^2} \rp} d z,
\\ 
\notag
\calT^{[4]}_{R,r}(n) &:=
- \sum_{\a_1,\a_2 =0}^{R-1} \sum_{k=1}^N k^{-\frac{3}{2}}  \!\!\!\!\!\!
\sum_{\substack{0 \leq h < k \\ \gcd(h,k) = 1}} \!\!\!\!\!\! \varphi_{R,r,h,k} ({\bm \a}) 
\int_{z_1}^{z_2} \!
\frac{f_{R,{\bm \a}} \!\lp \frac{h'}{k} + \frac{i}{z} \rp}{\sqrt{z} \eta \!\lp \frac{h'}{k} + \frac{i}{z} \rp}
e^{-2\pi i  n_s \lp \frac{h}{k} + \frac{iz}{k^2} \rp} d z,
\\
\label{eq:That5_definition}
\calT^{[5]}_{R,r}(n) &:=
- \sum_{\a_1,\a_2 =0}^{R-1} \sum_{k=1}^N k^{-\frac{3}{2}}  \!\!\!\!\!\!
\sum_{\substack{0 \leq h < k \\ \gcd(h,k) = 1}} \!\!\!\!\!\! \varphi_{R,r,h,k} ({\bm \a}) 
\int_{z_1}^{z_2} \!
\frac{\mathcal{I}_{R,{\bm \a},\frac{h'}{k}} \!\lp \frac{h'}{k} + \frac{i}{z} \rp}{\sqrt{z} \eta \!\lp \frac{h'}{k} + \frac{i}{z} \rp}
e^{-2\pi i  n_s \lp \frac{h}{k} + \frac{iz}{k^2} \rp} d z.
\end{align}
We start with the most complicated contribution $\calT^{[5]}_{R,r}(n)$ coming from the false-indefinite components of the generating function, for which we give a more detailed account following \cite{BCN}. To state our result, we define the Kloosterman sum associated with $\varphi$ from \eqref{eq:nu_varphi_gamma_definition} as
\begin{equation}\label{eq:Kloosterman5_definition}
K_{R,r,\bm{\a},\bm{\k}}^{[5]} (n,k)
 := 
\sum_{\substack{0 \leq h < k \\ \gcd(h,k) = 1}}
\!\!\!\!\!\!
\varphi_{R,r,h,k} \!\lp {\bm \a} \rp 
e^{ \frac{2\pi ih'R}{k}  \left(\k_1+\frac{\a_1}{R}\right)\left(\k_2+\frac{\a_2}{R}\right) -\frac{\pi i h'}{12k} - \frac{2\pi i  n_s h}{k}}.
\end{equation}
We also define the integral kernel
\begin{equation}\label{eq:phi_function_definition}
\Phi_{R,k,\bm{\l}} (w) := \sum_{\substack{\bm{n} \in \IZ^2 + \frac{\bm{\l}}{Rk} \\ n_1,n_2 \neq 0}} \frac{w}{Rk^2 n_1n_2 (w-Rk^2n_1n_2)},
\end{equation}
which is a meromorphic function on $\IC$ with simple poles contained in $( k \IZ + \frac{\l_1 \l_2}{R} ) \setminus \{0\}$.

\begin{lem}\label{lem:That_5_asymptotic_result}
If $n,R \in \IN$, $24\nmid R$, and $r \in \{1,2,\ldots, R\}$, then, as $n \to \infty$, we have 
\begin{multline*}
\calT^{[5]}_{R,r}(n) = \frac{1}{(24n_s)^{\frac{1}{4}}} 
\sum_{\a_1,\a_2 =0}^{R-1} \sum_{k=1}^{\lfloor \sqrt{n} \rfloor}
\sum_{\k_1,\k_2 \pmod{k}}
\frac{K_{R,r,\bm{\a},\bm{\k}}^{[5]} (n,k)}{k} 
\\ \times
\operatorname{PV} \int_0^{\frac{1}{24}} 
\Phi_{R,k,R \bm{\k}+\bm{\a}} (t) (1-24t)^{\frac{1}{4}} 
I_{\frac{1}{2}}\! \lp \frac{\pi}{k} \sqrt{ \frac{2n_s}{3} (1-24t)} \rp dt
+ O_R\!\lp n^{\frac{1}{4}} \log (n) \rp .
\end{multline*}
\end{lem}
\begin{proof}
We start with \eqref{eq:That5_definition} and use \eqref{eq:calI_decomposition_principal_nonprincipal} with $d = \frac{1}{24}$ (noting that $24 \nmid R$) to split
\begin{multline*}
\frac{\mathcal{I}_{R,{\bm \a},\frac{h'}{k}} \!\lp \frac{h'}{k} + \frac{i}{z} \rp}{\eta \!\lp \frac{h'}{k} + \frac{i}{z} \rp}
=
\zeta_{24k}^{-h'} \mathcal{I}^*_{R,{\bm \a},\frac{h'}{k},\frac{1}{24}} \!\lp \frac{h'}{k}+\frac{i}{z} \rp 
\\
+
\mathcal{I}_{R,{\bm \a},\frac{h'}{k}} \!\lp \frac{h'}{k} + \frac{i}{z} \rp
\lp \frac{1}{\eta \!\lp \frac{h'}{k} + \frac{i}{z} \rp} 
- e^{-\frac{\pi i}{12} \lp \frac{h'}{k} + \frac{i}{z} \rp} \rp
+ \zeta_{24k}^{-h'} 
\mathcal{I}^e_{R,{\bm \a},\frac{h'}{k},\frac{1}{24}} \!\lp \frac{h'}{k}+\frac{i}{z} \rp .
\end{multline*}
Here, the first term gives the principal part and the remaining ones give the nonprincipal part.
Using Lemma~\ref{lem:false_indef_part_nonprincipal} with $d=0$ and $d = \frac{1}{24}$, the nonprincipal part can be bounded as $\ll_R 1$ for $\re(\frac{1}{z})\geq 1$ (that is on and within the standard circle with center at $\frac{1}{2}$ and radius $\frac{1}{2}$). For the contribution of this part, we take the integral over the chord $\mathcal{C} (z_1,z_2)$ from $z_1$ to $z_2$, where we recall that for any $z \in \mathcal{C} (z_1,z_2)$ one has the bounds (see e.g.~Section~6 of \cite{BCN} for details)
\begin{equation*}
\frac{k^2}{2N^2} \leq \re (z) < \frac{2k^2}{N^2}, \qquad
\frac{k^2}{2N^2} \leq |z| < \frac{\sqrt{2}k}{N}, \andd
\mathrm{length} (\mathcal{C} (z_1,z_2)) < \frac{2 \sqrt{2} k}{N}.
\end{equation*}
Correspondingly, on $\mathcal{C} (z_1,z_2)$ and for any $n \in \mathbb{N}$, we have
$e^{-2\pi i  n_s  (\frac{h}{k} + \frac{iz}{k^2}) } \ll 1$ because $N = \lfloor \sqrt{n} \rfloor$.
Using these facts along with the trivial bound $|\varphi_{R,r,h,k} ({\bm \a})| \leq 1$ then yields
\begin{multline}\label{eq:That5_mid_point_v1}
\calT^{[5]}_{R,r}(n) =
\sum_{\a_1,\a_2 =0}^{R-1} \sum_{k=1}^N k^{-\frac{3}{2}}  \!\!\!\!\!\!
\sum_{\substack{0 \leq h < k \\ \gcd(h,k) = 1}} \!\!\!\!\!\! 
\varphi_{R,r,h,k} (\bm{\a}) 
\zeta_{24k}^{-h'-24n_sh}
\\[-7pt]
\times
\int_{1- \frac{i k_1}{k}}^{1+ \frac{i k_2}{k}}  \! z^{-\frac{3}{2}} 
\mathcal{I}^*_{R,\bm{\a},\frac{h'}{k},\frac{1}{24}} \!\lp \frac{h'}{k}+iz \rp
e^{ \frac{2 \pi n_s}{k^2 z}} dz +O_R\!\left(n^{\frac{1}{4}}\right)\!.
\end{multline}

We next recall the definition of $\mathcal{I}^*$ in \eqref{eq:false_indef_part_principal_defn}. Splitting off the finitely many terms with~$|Rn_1 n_2| < 1$ (where the principal value integral is needed), we can bound the remaining terms uniformly in the integration variable by $n_1^{-2}n_2^{-2}$. This allows us to switch the order of summation and integral. 
Letting $\bm{n} \mapsto k \bm{n}$ with 
$\bm{n} \in \IZ^2 + \frac{R \bm{\k} +\bm{\a}}{Rk}$
and $\k_1,\k_2$ running $\!\!\pmod{k}$ while recalling the definition of~$\Phi$ from \eqref{eq:phi_function_definition}, we then find
\begin{equation*}
\mathcal{I}^*_{R,\bm{\a},\frac{h'}{k},\frac{1}{24}} \!\lp \frac{h'}{k}+iz \rp
=
\frac{1}{2\pi i} \sum_{\k_1,\k_2 \pmod{k}} \!\!\!\!\!\!\!\!
e^{\frac{2\pi iRh'}{k} \lp \k_1 + \frac{\a_1}{R} \rp \lp \k_2 + \frac{\a_2}{R} \rp}
\mathrm{PV}\!\int_0^{\frac{1}{24}} 
e^{- 2 \pi z \left(t-\frac{1}{24}\right)} \Phi_{R,k,R \bm{\k} +\bm{\a}} (t) dt.
\end{equation*}
Inserting this into \eqref{eq:That5_mid_point_v1}, interchanging the two integrals, and recalling the definition of the Kloosterman sum \eqref{eq:Kloosterman5_definition}, we get
\begin{multline*}
\calT^{[5]}_{R,r}(n) =
\frac{1}{2\pi i}
\sum_{\a_1,\a_2 =0}^{R-1} \sum_{k=1}^N 
\sum_{\k_1,\k_2 \pmod{k}} \!\!\!\!\!\!
\frac{K_{R,r,\bm{\a},\bm{\k}}^{[5]} (n,k)}{k^{\frac{3}{2}}}
\\ \times
\mathrm{PV}\!\int_0^{\frac{1}{24}} 
\Phi_{R,k,R \bm{\k} +\bm{\a}} (t) 
\int_{1-\frac{ik_1}{k}}^{1+ \frac{ik_2}{k}}  \! z^{-\frac{3}{2}} 
e^{2 \pi \left( \frac{n_s}{k^2 z} + \lp \frac{1}{24} -t \rp z \right)} 
dz dt+ O_R\!\left(n^{\frac{1}{4}}\right)\!.
\end{multline*}
We follow the arguments of \cite{BCN} and define the integration paths shown in Figure~\ref{fig:contour_bessel_error} with $\mathcal{C} := \mathcal{C}_2 + \mathcal{C}_4 + \mathcal{C}_5 + \mathcal{C}_3 + \mathcal{C}_1$ giving a Hankel-type contour. Then we split the integral on $z$ into five pieces\footnote{
The integrals over $\mathcal{C}_3,\mathcal{C}_4$, and $\mathcal{C}_5$ yield entire functions of $t=w\in\mathbb{C}$ and those over $\mathcal{C}_1,\mathcal{C}_2$ are holomorphic for $\re (w) < \frac{1}{24}$ and continuous for $\re (w) \leq \frac{1}{24}$.
}
\begin{equation*}
\int_{1- \frac{ik_1}{k}}^{1+ \frac{ik_2}{k}}  \! z^{-\frac{3}{2}} 
e^{2 \pi \left( \frac{n_s}{k^2 z} + \lp \frac{1}{24} -t \rp z \right)} dz
=
\int_{\mathcal{C}} z^{-\frac{3}{2}}
e^{2 \pi \left( \frac{n_s}{k^2 z} + \lp \frac{1}{24} -t \rp z \right)} dz
- \sum_{j=1}^4 \int_{\mathcal{C}_j} z^{-\frac{3}{2}} 
e^{2 \pi \left( \frac{n_s}{k^2 z} + \lp \frac{1}{24} -t \rp z \right)} dz.
\end{equation*}

\begin{figure}[h!]
	\includegraphics[scale=0.3]{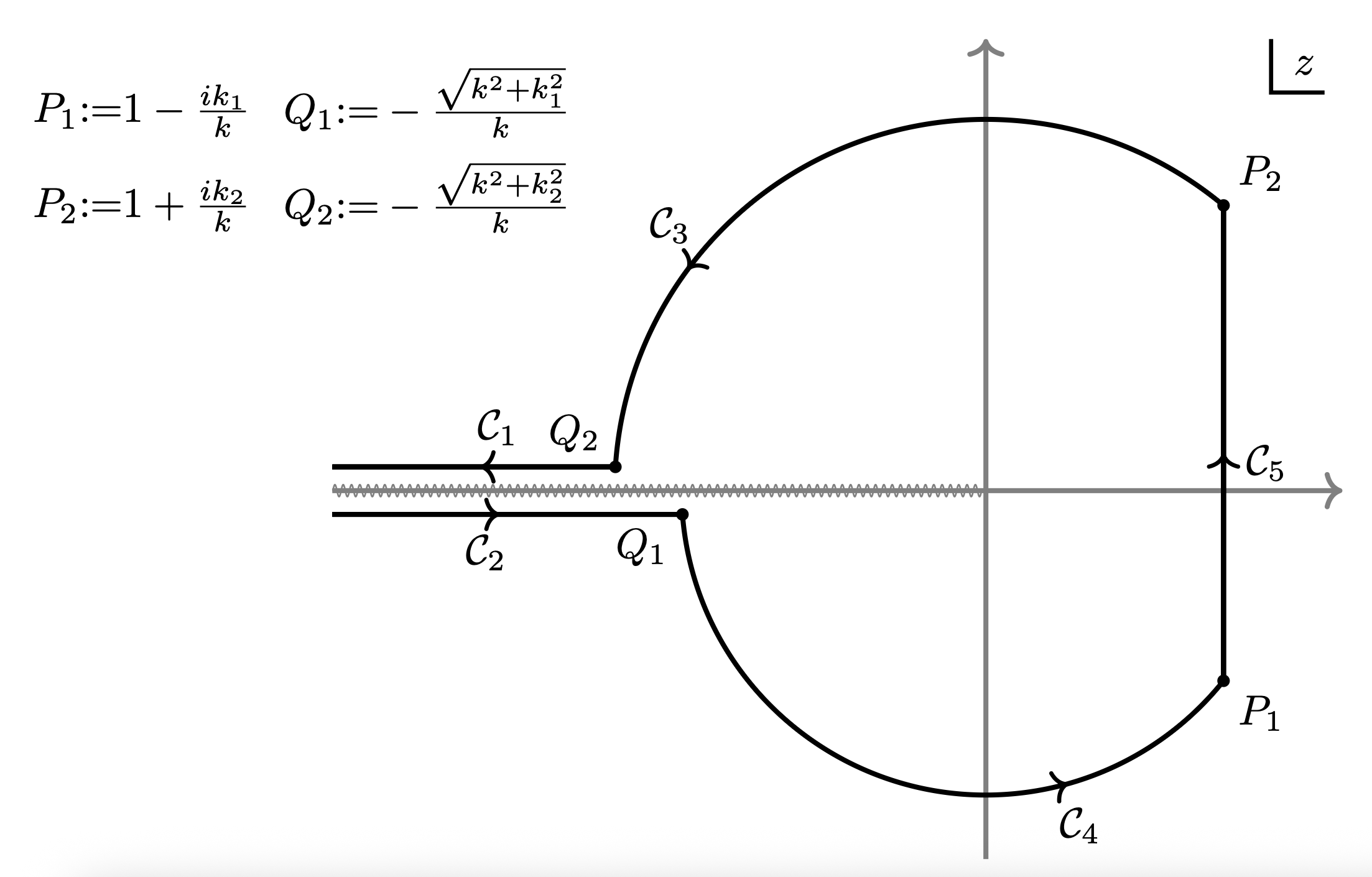}
	\vspace{-15pt}
	\caption{The integrations paths $\mathcal{C}_1,\mathcal{C}_2$, and $\mathcal{C}_5$ are straight lines and $\mathcal{C}_3$ and $\mathcal{C}_4$ are circular arcs centered at~$0$.}
\label{fig:contour_bessel_error}
\end{figure}

The integral over $\mathcal{C}$ can be evaluated in terms of $I$-Bessel functions using  \eqref{eq:I_Bessel_Hankel_integral} and its contribution to $\calT^{[5]}_{R,r}(n)$ yields the main term stated in the lemma. Our remaining task is to bound the contributions of the integrals over $\mathcal{C}_j$ with $j \in \{1,2,3,4\}$.
For this purpose, we write the kernel as 
\begin{equation}\label{eq:kernel_phi_pole_separation}
\Phi_{R,k,R \bm{\k} +\bm{\a}} (t) 
=:
\Phi^*_{R,k,R \bm{\k} +\bm{\a}} (t) 
+ 
\frac{d_{R,k,R \bm{\k} +\bm{\a}} t}{t_{R,k,R \bm{\k} +\bm{\a}} (t-t_{R,k,R \bm{\k} +\bm{\a}}) }
\end{equation}
with the second term splitting in \eqref{eq:phi_function_definition} the only possible term that may give a pole in $[0,\frac{1}{24}]$ off, so that $t_{R,k,R \bm{\k} +\bm{\a}}$ is the  unique value of $Rk^2n_1n_2$ in $[0,\frac{1}{24}]$ if such a value exists and $d_{R,k,R\bm{\k}+\bm{\a}}$ equals one if this summand exists and zero otherwise. We have (noting  $24 \nmid R$) 
\begin{equation}\label{eq:That5_delta_epsilon_R_definition_pole_location}
0 < t_{R,k,R \bm{\k} +\bm{\a}} < \frac{1}{24} \andd
0 < t_{R,k,R \bm{\k} +\bm{\a}}^{-1} \leq R .
\end{equation}
Also observe that for $t \in [0,\frac{1}{24}]$ we have
\begin{equation}\label{eq:That5_Phi_d_sum_bounds}
\sum_{\a_1,\a_2 =0}^{R-1} \sum_{\k_1,\k_2 \pmod{k}}
|\Phi^*_{R,k,R \bm{\k} +\bm{\a}} (t)| \ll_R 1
\andd
\sum_{\a_1,\a_2 =0}^{R-1} 
\sum_{\k_1,\k_2 \pmod{k}} 
d_{R,k,R \bm{\k} +\bm{\a}} \ll_R 1.
\end{equation}
To handle the pole contribution in \eqref{eq:kernel_phi_pole_separation}, we recall from \eqref{eq:leading_exp_term_E_Psi_definitions} that $\e_R = \frac{1}{2} \mathrm{dist} (\frac{1}{24}, \frac{\IZ}{R}) $. We have
\begin{equation}\label{eq:That5_delta_epsilon_R_definition_pole_interval}
0 < 2 \e_R < \frac{1}{R}
\andd
[t_{R,k,R \bm{\k} +\bm{\a}} - 2 \e_R, t_{R,k,R \bm{\k} +\bm{\a}} + 2 \e_R] 
\subset \left[0,\frac{1}{24}\right].
\end{equation}
The principal value is only needed for the pole contribution for $|t-t_{R,k,R \bm{\k} +\bm{\a}}| \leq \e_R$. We estimate such terms by removing the simple pole from the Laurent expansion (with vanishing contribution in this symmetric interval) and use the mean value bound.
Employing these facts along with the trivial bound on the Kloosterman sum, we follow \cite{BCN} and bound the contributions of the two terms of \eqref{eq:kernel_phi_pole_separation} integrated over $\mathcal{C}_j$.
We estimate those as $\ll_R n^{\frac{1}{4}}$ except for the pole contribution in~\eqref{eq:kernel_phi_pole_separation} along $\mathcal{C}_3$ and $\mathcal{C}_4$, which we bound as $\ll_R n^{\frac{1}{4}} \log (n)$, and the lemma follows.
\end{proof}

We are now ready to prove our main result.

\begin{proof}[Proof of Theorem~\ref{thm:asymptotic}]
Note that $\calT^{[5]}_{R,r}(n)$ from the splitting \eqref{splitT} is handled in Lemma~\ref{lem:That_5_asymptotic_result}. The contribution $\calT^{[4]}_{R,r}(n)$ is the standard modular-type contribution coming from the (modular) indefinite theta components of the generating function. 
Using the separation of the principal part in Lemma~\ref{lem:indef_part_nonprincipal} and using the same arguments as in Lemma~\ref{lem:That_5_asymptotic_result} (throughout using the trivial Kloosterman sum bounds), we obtain (including the case $24 \mid R$)
\begin{multline}\label{eq:That4_asymptotic_result}
\calT^{[4]}_{R,r}(n) = \frac{2\pi}{(24n_s)^{\frac{1}{4}}} 
\sum_{\substack{\a_1,\a_2 \ge1\\\alpha_1\alpha_2<\frac{R}{24}}}
\lp 1 - \frac{24\a_1\a_2}{R} \rp^{\frac{1}{4}} \sum_{k=1}^{\lfloor \sqrt{n} \rfloor} \frac{K_{R,r,\bm{\a}}^{[4]} (n,k)}{k}
I_{\frac{1}{2}} \!\lp \frac{\pi}{k} \sqrt{\frac{2}{3} \! \lp 1 - \frac{24\a_1\a_2}{R} \rp n_s} \rp 
\\[-.75em]
+ O_R\!\lp n^{\frac{1}{4}} \rp ,
\end{multline}
where the associated Kloosterman sum is defined as
\begin{equation}\label{eq:Kloosterman4_definition}
K_{R,r,\bm{\a}}^{[4]} (n,k)
:= i \!\!\!\!\!\!
\sum_{\substack{0 \leq h < k \\ \gcd(h,k) = 1}}
\!\!\!\!\!\!
\frac{\varphi_{R,r,h,k} ({\bm \a})-\varphi_{R,r,h,k} ((R,R)-{\bm \a})}{2}
e^{-\frac{\pi i h'}{12k} \lp 1 - \frac{24 \a_1\a_2}{R} \rp - \frac{2\pi i  n_s h}{k}}.
\end{equation}
Using Lemma~\ref{lem:alpha_beta_bounds} and \eqref{eq:I_Bessel_Hankel_integral}, as $n \to \infty$, we similarly obtain (again also for $24 \mid R$)
\begin{align}\label{eq:That2_asymptotic_result}
\calT^{[2]}_{R,r}(n) &= \frac{2\pi}{(24n_s)^{\frac{1}{4}}} 
\sum_{k=1}^{\lfloor \sqrt{n} \rfloor} \frac{K_{R,r}^{[2]} (n,k)}{k^2}
I_{\frac{1}{2}} \!\lp \frac{\pi}{k} \sqrt{\frac{2n_s}{3}} \rp 
+ O_R\!\lp n^{\frac{3}{4}} \log (n) \rp,
\\
\label{eq:That3_asymptotic_result}
\calT^{[3]}_{R,r}(n) &= \frac{2\pi}{(24n_s)^{\frac{3}{4}}} 
\sum_{k=1}^{\lfloor \sqrt{n} \rfloor} \frac{K_{R,r}^{[3]} (n,k)}{k}
I_{\frac{3}{2}} \!\lp \frac{\pi}{k} \sqrt{\frac{2n_s}{3}} \rp 
+ O\!\lp n^{\frac{3}{4}} \rp ,
\end{align}
where we define the Kloosterman-like sums
\begin{equation}\label{eq:Kloosterman2_3_definition}
K_{R,r}^{[2]} (n,k) \! := 
\!\!\!\!\!\! \sum_{\substack{0 \leq h < k \\ \gcd(h,k) = 1}} \!\!\!\!\!\! \!\!
\nu_{h,k} 
A_{R,r,h,k}  
\zeta_{24k}^{-h'-24n_sh}
\ \mbox{ and } \
K_{R,r}^{[3]} (n,k) \! := 
\!\!\!\!\!\! \sum_{\substack{0 \leq h < k \\ \gcd(h,k) = 1}} \!\!\!\!\!\!\!\!
\nu_{h,k} 
B_{R,r,h,k} 
\zeta_{24k}^{-h'-24n_sh}.
\end{equation}
Finally, the contribution $\calT^{[1]}_{R,r}(n)$ is handled in the same way with \eqref{eq:I_Bessel_derivative_c_infty_integral} to give (again also for~$24 \mid R$)
\begin{equation}\label{eq:That1_asymptotic_result}
\!\!\calT^{[1]}_{R,r}(n) = 
\!\sum_{k=1}^{\lfloor \sqrt{n} \rfloor} \frac{ \d_{\gcd(R,k) \mid r} K^{[1]} (n,k)}{ \mathrm{lcm} (R,k) (24n_s)^{\frac{1}{4}}} \!
 \lp \log\!\lp \frac{2\sqrt{6n_s}}{k} \rp I_{\frac{1}{2}} \!\lp\frac{\pi}{k} \sqrt{\frac{2n_s}{3}} \rp \! - \mathbb I_{\frac{1}{2}} \!\lp \frac{\pi}{k} \sqrt{\frac{2n_s}{3}} \rp \!\rp
\!+ O\!\left(n^{\frac{1}{4}}\right)\!,
\end{equation}
where
\begin{equation}\label{eq:Kloosterman1_definition}
K^{[1]} (n,k) := \sum_{\substack{0 \leq h < k \\ \gcd(h,k) = 1}} \!\!\!\!\!\!
\nu_{h,k} \zeta_{24k}^{-h'-24n_sh}.
\end{equation}
Plugging all these results into \eqref{splitT} gives the theorem.
\end{proof}

\section{Proof of Corollaries \ref{cor:leading_exponential_asymptotic} and \ref{cor:asymptoticleading}}\label{sec:leading}

Here we further investigate the leading contributions in the asymptotics of $\calT_{R,r} (n)$ by proving Corollaries~\ref{cor:leading_exponential_asymptotic} and~\ref{cor:asymptoticleading}. 
We start with $\calT^{[5]}_{R,r}(n)$.

\begin{lem}\label{lem:leading_exponential_asymptotic_That_5}
Let $n,R \in \IN$ with $24 \nmid R$, and $r \in \{1,2,\ldots, R\}$. Then we have
\begin{equation}\label{eq:leading_exponential_asymptotic_That_5}
\calT^{[5]}_{R,r} (n)  \!
= \!
\frac{1}{2 \pi R \sqrt{2 n_s}}
\int_0^{\e_R} \Psi_{R,r} (t) 
 e^{\pi \sqrt{\frac{2n_s}{3} (1-24t)}} dt
+O_R \!\lp e^{ \pi \sqrt{\frac{2}{3}(1-24\varepsilon_R)n_s}} \rp \! .
\end{equation}
\end{lem}
\begin{proof}
By Lemma~\ref{lem:That_5_asymptotic_result} and \eqref{eq:I_Bessel_half_integral_examples}, as $n \to \infty$, we have
\begin{multline*}
\calT^{[5]}_{R,r}(n) 
= \frac{1}{\pi \sqrt{2 n_s}} 
\sum_{\a_1,\a_2 =0}^{R-1} \sum_{k=1}^{\lfloor \sqrt{n} \rfloor}
\sum_{\k_1,\k_2 \pmod{k}}
\frac{K_{R,r,\bm{\a},\bm{\k}}^{[5]} (n,k)}{\sqrt{k}} 
\\ \times
\operatorname{PV} \int_0^{\frac{1}{24}} 
\Phi_{R,k,R \bm{\k}+\bm{\a}} (t) 
\sinh\! \lp \frac{\pi}{k} \sqrt{\frac{2n_s}{3}(1-24t)} \rp dt
+ O_R\!\lp n^{\frac{1}{4}} \log (n) \rp .
\end{multline*}
Recall that $\Phi_{R,k,R \bm{\k}+\bm{\a}} (t)$ has at most one (simple) pole in $[0,\frac{1}{24}]$, denoted by $t_{R,k,R \bm{\k}+\bm{\a}}$, which is in the interior for $24 \nmid R$ (see the discussion around \eqref{eq:kernel_phi_pole_separation}).
Using \eqref{eq:kernel_phi_pole_separation} and \eqref{eq:That5_delta_epsilon_R_definition_pole_interval}, we write
\begin{multline*}
2 \operatorname{PV} \int_0^{\frac{1}{24}} 
\Phi_{R,k,R \bm{\k}+\bm{\a}} (t) 
\sinh\! \lp \frac{\pi}{k} \sqrt{\frac{2n_s}{3}(1-24t)} \rp dt
\\
=
J^{[1]}_{R,k,\bm{\k},\bm{\a}} (n) + J^{[2]}_{R,k,\bm{\k},\bm{\a}} (n)
+ \frac{d_{R,k,R \bm{\k}+\bm{\a}}}{t_{R,k,R \bm{\k}+\bm{\a}}}   
J^{[3]}_{R,k,\bm{\k},\bm{\a}} (n),
\end{multline*}
where we let $f_{n,k} (t) := t \sinh(\frac{\pi}{k} \sqrt{\frac{2n_s}{3}(1-24t)})$ and
\begin{align}
J^{[1]}_{R,k,\bm{\k},\bm{\a}} (n)
&:= 
\int_0^{\e_R} \Phi_{R,k,R \bm{\k} +\bm{\a}} (t) 
e^{\frac{\pi}{k} \sqrt{\frac{2n_s}{3}(1-24t)} } dt,
\notag \\
J^{[2]}_{R,k,\bm{\k},\bm{\a}} (n)
&:= 
\int_{\e_R}^{\frac{1}{24}} \Phi^*_{R,k,R \bm{\k} +\bm{\a}} (t)   
e^{\frac{\pi}{k} \sqrt{\frac{2n_s}{3}(1-24t)} } dt
-
\int_{0}^{\frac{1}{24}} \Phi^*_{R,k,R \bm{\k} +\bm{\a}} (t) 
e^{-\frac{\pi}{k} \sqrt{\frac{2n_s}{3}(1-24t)} } dt,
\notag \\
J^{[3]}_{R,k,\bm{\k},\bm{\a}} (n)
&:= 
2 \int_{t_{R,k,R \bm{\k} +\bm{\a}}-\e_R}^{t_{R,k,R \bm{\k} +\bm{\a}}+\e_R} 
\frac{f_{n,k} (t) - f_{n,k} (t_{R,k,R \bm{\k} +\bm{\a}})}{t-t_{R,k,R \bm{\k} +\bm{\a}}} dt
\notag \\ & \qquad
+
\lp \int_{\e_R}^{t_{R,k,R \bm{\k} +\bm{\a}}-\e_R}
+ \int_{t_{R,k,R \bm{\k} +\bm{\a}}+\e_R}^{\frac{1}{24}} \rp
\frac{t e^{\frac{\pi}{k} \sqrt{\frac{2n_s}{3}(1-24t)}}}{t-t_{R,k,R \bm{\k} +\bm{\a}}} dt
\notag \\ & \qquad \qquad 
-
\lp \int_{0}^{t_{R,k,R \bm{\k} +\bm{\a}}-\e_R}
+ \int_{t_{R,k,R \bm{\k} +\bm{\a}}+\e_R}^{\frac{1}{24}} \rp
\frac{t e^{-\frac{\pi}{k} \sqrt{\frac{2n_s}{3}(1-24t)}}}{t-t_{R,k,R \bm{\k} +\bm{\a}}} dt .\nonumber
\end{align}
Here, $J^{[3]}_{R,k,\bm{\k},\bm{\a}} (n)$ is defined assuming that $\Phi_{R,k,R \bm{\k} +\bm{\a}} (t)$ has a pole inside $[0,\frac{1}{24}]$, otherwise we take this term to be zero. Noting that $t \leq \frac{1}{24}$ and $24 \e_R \leq 1-24t \leq 1-24\e_R$ for $t \in [{t_{R,k,R \bm{\k} +\bm{\a}}-\e_R},\linebreak {t_{R,k,R \bm{\k} +\bm{\a}}+\e_R}]$ by \eqref{eq:That5_delta_epsilon_R_definition_pole_interval} and employing the mean value bound, we have
\begin{equation*}
\left| \frac{f_{n,k} (t) - f_{n,k} (t_{R,k,R \bm{\k} +\bm{\a}})}{t-t_{R,k,R \bm{\k} +\bm{\a}}} \right|
\leq 
\lp 1 + \frac{\pi}{12k} \sqrt{\frac{n_s}{\e_R}}  \rp e^{\frac{\pi}{k} \sqrt{\frac{2}{3}(1-24\varepsilon_R)n_s}}.
\end{equation*}
This then implies that 
\begin{equation*}
J^{[3]}_{R,k,\bm{\k},\bm{\a}} (n) \ll_R \sqrt{n_s} e^{\frac{\pi}{k} \sqrt{\frac{2}{3}(1-24\varepsilon_R)n_s}} .
\end{equation*}
With \eqref{eq:That5_delta_epsilon_R_definition_pole_location}, we then bound the contribution of $J^{[3]}_{R,k,\bm{\k},\bm{\a}} (n)$ to $\calT^{[5]}_{R,r}(n)$ as $\ll_R e^{\pi \sqrt{\frac{2}{3} (1- 24\e_R)n_s}}$, which is accounted for by the error term in the lemma. This is also true for the contribution of $J^{[2]}_{R,k,\bm{\k},\bm{\a}} (n)$ by~\eqref{eq:That5_Phi_d_sum_bounds}.
The same arguments can also be used to bound the contribution of $J^{[1]}_{R,k,\bm{\k},\bm{\a}} (n)$ to $\calT^{[5]}_{R,r}(n)$ for $k \geq 2$ by
$\ll_R e^{\pi \sqrt{\frac{n_s}{6}}}$, which is again accounted for by the error term in the lemma.
This leaves the contribution of $J^{[1]}_{R,1,\bm{\k},\bm{\a}} (n)$, so that
\begin{equation*}
\calT^{[5]}_{R,r}(n) = \frac{1}{2\pi \sqrt{2n_s}} 
\sum_{\a_1,\a_2 =0}^{R-1} 
K_{R,r,\bm{\a},\bm{0}}^{[5]} (n,1)
\int_0^{\e_R} \Phi_{R,1,\bm{\a}} (t) e^{\pi \sqrt{\frac{2n_s}{3}(1-24t)} } dt
+O_R \!\lp e^{\pi \sqrt{\frac{2}{3} (1- 24\e_R)n_s}} \rp  \! .
\end{equation*}
The lemma follows once we note $K_{R,r,\bm{\a},\bm{0}}^{[5]} (n,1) = \frac{1}{R} \zeta_R^{-r\a_1}$ by \eqref{eq:psi_definition}, \eqref{eq:nu_varphi_gamma_definition}, \eqref{eq:Kloosterman5_definition} and the fact that
$\sum_{\a_1,\a_2 =0}^{R-1} 
\zeta_R^{-r\a_1} \Phi_{R,1,\bm{\a}} (t)=\Psi_{R,r} (t)$
by the definitions of $\Psi$ and $\Phi$ in~\eqref{eq:leading_exp_term_E_Psi_definitions}, \eqref{eq:phi_function_definition},  respectively.
\end{proof}
We are now ready to prove Corollary~\ref{cor:leading_exponential_asymptotic}.

\begin{proof}[Proof of Corollary~\ref{cor:leading_exponential_asymptotic}]
With $\calT^{[5]}_{R,r}(n)$ examined in Lemma~\ref{lem:leading_exponential_asymptotic_That_5}, we study the remaining terms in~\eqref{splitT}. We start with the asymptotics in \eqref{eq:That1_asymptotic_result} for~$\calT^{[1]}_{R,r}(n)$, 
\eqref{eq:That2_asymptotic_result} for~$\calT^{[2]}_{R,r}(n)$, and \eqref{eq:That3_asymptotic_result} for~$\calT^{[3]}_{R,r}(n)$.
Using the explicit characterization of the Bessel functions in \eqref{eq:I_Bessel_half_integral_examples}, \eqref{eq:I_Bessel_order_der_1_2_example}, and \eqref{eq:exponential_integral_term_asymptotics},
and the definitions of Kloosterman sums \eqref{eq:Kloosterman2_3_definition} and \eqref{eq:Kloosterman1_definition} (also see Lemma~\ref{lem:alpha_beta_bounds}), we find
\begin{multline}\label{eq:leading_exponential_asymptotic_That_1_2_3}
\sum_{j=1}^3 \calT^{[j]}_{R,r} (n)  
= 
\frac{e^{\pi \sqrt{\frac{2n_s}{3}}}}{4 \pi R \sqrt{2 n_s}}
\Bigg(\!
\log (n_s) - \log \!\lp \frac{\pi^2}{6} \rp 
- 2 \psi\!\lp \frac{r}{R} \rp - 2 \log (R)+ \frac{\pi (2r-R)}{2\sqrt{6n_s}}
+ \frac{R-2r}{4 n_s}
\\ 
+ 2 \CE\!\lp 2 \pi \sqrt{\frac{2n_s}{3} } \rp\!\!
\Bigg)
+O_R\!\lp e^{\pi \sqrt{\frac{n_s}{6}}} \rp \! .
\end{multline}
Similarly, using \eqref{eq:That4_asymptotic_result}, \eqref{eq:I_Bessel_half_integral_examples}, and \eqref{eq:Kloosterman4_definition} while noting $K_{R,r,\bm{\a}}^{[4]} (n,1) = \frac{1}{R} \sin ( \frac{2 \pi r\a_1}{R} )$, we find that
\begin{equation}\label{eq:leading_exponential_asymptotic_That_4}
\calT^{[4]}_{R,r} (n) = 
\frac{1}{R \sqrt{2n_s}}  
\sum_{\substack{\a_1,\a_2 \ge1 \\ \a_1\a_2 < \frac{R}{32}}}
\sin \!\lp \frac{2 \pi r \a_1}{R} \rp
e^{\pi \sqrt{\frac{2}{3} \! \lp 1 - \frac{24\a_1\a_2}{R} \rp n_s}} 
+O_R\!\lp e^{\pi \sqrt{\frac{n_s}{6}}} \rp \!
\ll_R
e^{\pi \sqrt{\frac{2}{3} (1-24 \e_R) n_s}} .
\end{equation}
Combining this result with Lemma~\ref{lem:leading_exponential_asymptotic_That_5} and~\eqref{eq:leading_exponential_asymptotic_That_1_2_3} yields the corollary.
\end{proof}

Our final goal is to determine the leading exponential term in $\calT_{R,r}(n)$ more explicitely.

\begin{proof}[Proof of Corollary~\ref{cor:asymptoticleading}]
We start with the integral in Corollary~\ref{cor:leading_exponential_asymptotic} and make the change of variables $u := \frac{1-\sqrt{1-24t}}{12}$ as in the proof of Proposition 7.2 in \cite{BCN} to write
\begin{equation*}
\int_0^{\e_R} \Psi_{R,r} (t) 
e^{- \pi \sqrt{\frac{2n_s}{3}} \left(1-\sqrt{1-24t}\right)} dt
=
\int_0^{\d_R} f_{R,r} (u)
e^{-4 \pi \sqrt{6 n_s} u} du,
\end{equation*}
where
\begin{equation*}
\d_R := \frac{1-\sqrt{1-24\e_R}}{12} 
\andd
f_{R,r} (u) := (1-12u) \Psi_{R,r} (u(1-6u)) .
\end{equation*}
Noting that $f_{R,r}$ is smooth on $[0,\d_R]$ and that $f_{R,r} (0) = 0$, we use its Taylor expansion around the point $u=0$ to find (for $L \in \IN_{\geq 2}$)
\begin{equation}\label{eq:leading_exp_part_integral_part1}
\int_0^{\e_R} \Psi_{R,r} (t) 
e^{- \pi \sqrt{\frac{2n_s}{3}} \left(1-\sqrt{1-24t}\right)} dt
=
\sum_{\ell=2}^{L} \frac{f_{R,r}^{(\ell-1)} (0)}{\left(4 \pi \sqrt{6n_s}\right)^\ell} 
+ O_{L,R} \!\lp n^{-\frac{L+1}{2}} \rp .
\end{equation}
By induction one can show that, for $\ell\in\IN$, we have
\begin{equation}\label{eq:leading_exp_part_integral_part2}
f^{(\ell-1)}_{R,r} (u)
= \sum_{j=\left\lceil \frac{\ell}{2} \right\rceil}^{\ell} 
\frac{(-6)^{\ell-j} \ell!}{(\ell-j)! (2j-\ell)!}  
(1-12u)^{2j-\ell}  
\Psi_{R,r}^{(j-1)} (u(1-6u)).
\end{equation}
Using \eqref{eq:leading_exp_term_E_Psi_definitions} and expanding the denominator, we have, for $|t| < \frac{1}{R}$,
\begin{equation*}
\Psi_{R,r} (t) = - \sum_{j\geq 0} R^{j+2} t^{j+1} 
\sum_{n_1,n_2\in\Z\setminus\{0\}} 
\frac{\zeta_R^{rn_1}}{(n_1 n_2)^{j+2}}.
\end{equation*}
By \eqref{eq:Bernoulli_polynomial_identity}, we then evaluate
\begin{equation*}
\Psi_{R,r}^{(2j-1)} (0) = - \frac{\left(4 \pi^2 R\right)^{2j}}{2j (2j)!} 
B_{2j}  B_{2j} \!\lp \frac{r}{R} \rp
\mbox{ for } j \in \IN
\andd
\Psi_{R,r}^{(2j)} (0) = 0 \mbox{ for } j \in \IN_0 .
\end{equation*}
Using this in \eqref{eq:leading_exp_part_integral_part1} and \eqref{eq:leading_exp_part_integral_part2} then yields (for $L \in \IN_{\geq 2}$)
\begin{equation*}
\int_0^{\e_R} \Psi_{R,r} (t) 
e^{- \pi \sqrt{\frac{2n_s}{3}} \left(1-\sqrt{1-24t}\right)} dt
=
\sum_{\ell=2}^{L} \frac{c_{R,r,\ell}}{n_s^{\frac{\ell}{2}}} 
+ O_{L,R} \!\lp n^{-\frac{L+1}{2}} \rp ,
\end{equation*}
where
\begin{equation*}
c_{R,r,\ell} :=(-1)^{\ell+1}  \ell!
\lp \frac{\sqrt{3}}{2\sqrt{2}\pi} \rp^{\ell} 
\sum_{j = \left\lceil \frac{\ell}{4} \right\rceil}^{\left\lfloor \frac{\ell}{2} \right\rfloor}
\frac{1}{2j} \! \lp \frac{2\pi^2 R}{3} \rp^{2j} \!\!
\frac{B_{2j} B_{2j}\!\lp \frac{r}{R} \rp}{(\ell-2j)! (4j-\ell)! (2j)!} .
\end{equation*}
The corollary then follows using Corollary~\ref{cor:leading_exponential_asymptotic} and \eqref{eq:exponential_integral_term_asymptotics}.
\end{proof}
We are now ready to prove Corollary~\ref{cor:leading_exponential_asymptotic_antisymmetric}.

\begin{proof}[Proof of Corollary~\ref{cor:leading_exponential_asymptotic_antisymmetric}]
By \eqref{eq:psi_definition} and \eqref{eq:nu_varphi_gamma_definition} we have $\varphi_{R,r,h,k} (\bm{\a})  = \varphi_{R,R-r,h,k} (-\bm{\a})$ and by \eqref{eq:calI_definition} we have $\mathcal{I}_{R,\bm{\a},-\frac{d}{c}} (\t) = \mathcal{I}_{R,-\bm{\a},-\frac{d}{c}} (\t)$, where both are periodic in $\a_1,\a_2$ with period $R$. Therefore, by~\eqref{eq:That5_definition}, we have $\calT^{[5]}_{R,r}(n) = \calT^{[5]}_{R,R-r}(n)$, which hence does not contribute to $\mathcal{T}_{R,r}(n)-\mathcal{T}_{R,R-r}(n)$. The corollary follows by \eqref{eq:leading_exponential_asymptotic_That_1_2_3} and \eqref{eq:leading_exponential_asymptotic_That_4} (for which we note that $24 \nmid R$ is not required).
\end{proof}

\end{document}